\documentclass[11pt]{amsart}

\usepackage[article]{robertpreamble}
\usepackage[margin=1in]{geometry}

\begin{document}


\title{Sharp Gaussian Isoperimetry along a Ricci Flow}
\author{Robert Koirala}
\address{Department of Mathematics, University of California San Diego}
\email{rkoirala@ucsd.edu}
\date{\today}

\begin{abstract}
    We prove the sharp Gaussian isoperimetric inequality for conjugate heat-kernel measures along a Ricci flow via a monotonicity formula. As consequences, we obtain the exact Gaussian enlargement theorem and a Gaussian-quantile two-set concentration estimate. In particular, this recovers the exponential concentration estimate of Hein--Naber from a sharper isoperimetric profile. We also derive Gaussian rearrangement inequalities, recover the sharp Hein--Naber log-Sobolev inequality, and identify the universal Gaussian-model constants in Bamler's \(L^p\)-Poincaré inequalities. Further applications include Gaussian-profile localization near Bamler's \(H_n\)-centers, convex-order and moment estimates for logarithmic derivatives of the conjugate heat kernel, reverse hypercontractivity, entropy-regular profile stability, and a path-space Bobkov inequality.
\end{abstract}

\maketitle 
\tableofcontents


\section{Introduction}

The Gaussian isoperimetric inequality, proved independently by Borell \cite{MR399402} and Sudakov--Tsirelson \cite{MR365680}, is one of the rigid features of the Euclidean heat kernel; see also \cite[Theorem~2.4.6]{MR4573029}. It says that half-spaces minimize Gaussian perimeter among all sets of prescribed Gaussian measure. More precisely, let \(\gamma_\tau\) be the centered Gaussian measure on \(\mathbb R^n\) with covariance \(2\tau\operatorname{Id}_n\), namely \(d\gamma_\tau(x) = (4\pi\tau)^{-n/2}e^{-|x|^2/4\tau}\,dx.\) Then every finite-perimeter set \(E\subset\mathbb R^n\) satisfies
\begin{align}
    \operatorname{Per}_{\gamma_\tau}(E)
    \ge
    \frac1{\sqrt{2\tau}}I(\gamma_\tau(E)),
    \label{eq:euclidean-GI}
\end{align}
where
\begin{align*}
    \Phi(r)=\frac1{\sqrt{2\pi}}\int_{-\infty}^r e^{-z^2/2}\,dz,
    \qquad
    \varphi=\Phi',
    \qquad
    I=\varphi\circ\Phi^{-1}.
\end{align*}
If \(H_c=\{x_1<c\}\) is a half-space with \(\gamma_\tau(H_c)=a\), then \(c=\sqrt{2\tau}\,\Phi^{-1}(a)\), and \(\operatorname{Per}_{\gamma_\tau}(H_c)=\frac1{\sqrt{2\tau}}I(a).\) Thus a half-space of \(\gamma_\tau\)-measure \(a\) optimizes \eqref{eq:euclidean-GI} among all sets of \(\gamma_\tau\)-measure \(a\).

The purpose of this paper is to prove the corresponding sharp inequality for conjugate heat-kernel measures along a Ricci flow. We then use this isoperimetric profile to revisit several estimates in the heat-kernel geometry of Ricci flow, developed by Hein--Naber and Bamler after Perelman \cite{perelman2002entropy,MR3245102,bamler2020entropy}. These estimates play a central role in Bamler's structure theory of Ricci flow \cite{bamler2020structure}.

Let \((M^n,g_t)_{t\in[s,t_0]}\) be a smooth closed Ricci flow \cite{MR664497},
\begin{align*}
    \partial_tg_t=-2\operatorname{Ric}_{g_t}.
\end{align*}
Fix a spacetime point \((x_0,t_0)\), put \(\tau=t_0-s\), and let
\begin{align*}
    d\nu(y)=d\nu_{x_0,t_0;s}(y)
    =
    K(x_0,t_0;y,s)\,dg_s(y)
\end{align*}
be the conjugate heat-kernel probability measure based at \((x_0,t_0)\) and evaluated at time \(s\).

Our main result is the following sharp Gaussian isoperimetric inequality. For a Borel set \(E\subset M\), define its weighted perimeter with respect to \(\nu\) by relaxation:
\begin{align*}
    \operatorname{Per}_{\nu}(E)
    =
    \inf
    \left\{
    \liminf_{j\to\infty}
    \int_M |\nabla f_j|_{g_s}\,d\nu
    \right\},
\end{align*}
where the infimum is taken over smooth functions \(0\le f_j\le1\) such that \(f_j\to\mathbf 1_E\) in \(L^1(\nu)\). If \(E\) has finite perimeter with respect to \(g_s\), then, since \(K(x_0,t_0;\cdot,s)\) is smooth and positive, this agrees with
\begin{align*}
    \operatorname{Per}_{\nu}(E)
    =
    \int_{\partial^*E}
    K(x_0,t_0;y,s)\,d\mathcal H^{n-1}_{g_s}(y).
\end{align*}

\begin{theorem}\label{thm:intro-GI}
    For every Borel set \(E\subset M\),
    \begin{align}
            \operatorname{Per}_{\nu}(E)
            \ge
            \frac1{\sqrt{2(t_0-s)}}I(\nu(E)).
            \label{eq:GI}
    \end{align}
    Equivalently, the weighted isoperimetric profile of \((M,g_s,\nu_{x_0,t_0;s})\) is bounded below by the Gaussian profile of variance \(2(t_0-s)\).
\end{theorem}

The constant is sharp. In the static Euclidean flow, \(\nu_{x_0,t_0;s}\) is the Gaussian measure with covariance \(2(t_0-s)\operatorname{Id}_n\), and equality is attained by half spaces.

The proof is inspired by the Bakry--Ledoux proof \cite{MR1374200} of Bobkov inequality \cite{MR1428506}. For smooth \(f:M\to[0,1]\), we prove
\begin{align}
        I\left(\int_M f\,d\nu\right)
        \le
        \int_M
        \sqrt{I(f)^2+2\tau|\nabla f|_{g_s}^2}\,d\nu.
        \label{eq:intro-bobkov}
\end{align}
The fundamental computation is the Ricci-flow Bochner cancellation: if \(\Box\coloneqq \partial_t-\Delta_{g_t}\) and \(\Box u=0\), then
\begin{align}
    \Box|\nabla u|^2=-2|\nabla^2u|^2.
    \label{eq:lovely-bochner}
\end{align}
The Ricci term in the usual Bochner formula is exactly cancelled by the metric variation. This cancellation underlies several sharp estimates in the heat-kernel geometry of Ricci flow, including the Hein--Naber log-Sobolev and concentration estimates \cite{MR3245102} and Bamler's sharp heat-flow gradient estimate \cite[Theorem~4.1]{bamler2020entropy}. The point here is that the same cancellation underlies a monotonicity formula for Bobkov's functional. Indeed, if \(u(\cdot,t)\) denotes the forward evolution of \(f\) from time \(s\) to time \(t\) under the heat equation coupled to the Ricci flow, and
\begin{align*}
    Q(\cdot,t)
    =
    \sqrt{
    I(u)^2+2(t_0-t)|\nabla u|_{g_t}^2
    },
\end{align*}
then the Ricci-flow Bochner cancellation \eqref{eq:lovely-bochner}, together with the time-dependent coefficient \(2(t_0-t)\), gives \(\Box Q\le0\). Hence
\begin{align*}
    t\longmapsto \int_M Q(\cdot,t)\,d\nu_{x_0,t_0;t}
\end{align*}
is nonincreasing. Evaluating this monotonicity at \(t=t_0\) and \(t=s\) gives \eqref{eq:intro-bobkov}, and Theorem~\ref{thm:intro-GI} follows by approximating characteristic functions.

The first consequence is the exact Gaussian enlargement estimate. If \(E_0=E\) and, for \(r>0\), \(E_r=\{y:d_{g_s}(y,E)<r\},\) then
\begin{align*}
    \nu(E_r)
    \ge
    \Phi\left(\Phi^{-1}(\nu(E))+\frac{r}{\sqrt{2\tau}}\right)
    \qquad
    \text{for all }r\ge0.
\end{align*}
Consequently, for Borel sets \(A,B\subset M\) with \(d_{g_s}(A,B)>0\),
\begin{align*}
    \Phi^{-1}(\nu(A))+\Phi^{-1}(\nu(B))
    \le
    -\frac{d_{g_s}(A,B)}{\sqrt{2\tau}}.
\end{align*}
In particular,
\begin{align*}
    \nu(A)\nu(B)
    \le
    \Phi\left(-\frac{d_{g_s}(A,B)}{2\sqrt{2\tau}}\right)^2
    \le
    \exp\left(-\frac{d_{g_s}(A,B)^2}{8\tau}\right),
\end{align*}
which recovers the exponential two-set concentration estimate of Hein--Naber \cite[Theorem~1.13]{MR3245102}. Thus the Gaussian isoperimetric profile gives a quantile refinement of Hein--Naber concentration estimate. In Section~\ref{sec:poincare}, we also show that the same profile implies a Gaussian Pólya--Szegő rearrangement inequality and hence gives another proof of the sharp Hein--Naber log-Sobolev inequality \cite[Theorem~1.10]{MR3245102}. In this sense, Theorem~\ref{thm:intro-GI} provides a common isoperimetric source for the sharp concentration, log-Sobolev, and Poincaré estimates considered below.

We then record several further consequences in Ricci-flow heat-kernel geometry. Around Bamler's \(H_n\)-centers, the one-sided concentration estimate gives Gaussian-profile tails. The logarithmic derivative of the conjugate heat kernel satisfies a convex-order domination by the one-dimensional Gaussian model, giving exact constants for directional logarithmic derivative moments. The same profile also gives support-sensitive Cheeger and Faber--Krahn estimates, Gaussian-model \(L^p\)-Poincaré inequalities, and the reverse hypercontractive counterpart to Bamler's \(p>1\) hypercontractivity theorem \cite[Theorem~12.1]{bamler2020entropy}.

Finally, we include two extensions of the main argument. First, in the entropy-regular regime, Bamler's \(\varepsilon\)-regularity theorem allows one to compare the heat-kernel weighted space to Euclidean Gaussian space. This gives an almost sharp upper bound for the weighted isoperimetric profile and a half-space stability statement for almost extremizers. Second, the finite-dimensional Bobkov inequality tensorizes along parabolic path space. Using the Haslhofer--Naber path-space gradient estimate \cite{MR3790068}, this yields a path-space Bobkov inequality and the corresponding Gaussian isoperimetric inequality on path space.

The paper is organized as follows. In Section~\ref{sec:bobkov}, we prove Bobkov inequality, deduce Theorem~\ref{thm:intro-GI}, and discuss the rigidity mechanism behind equality. In Section~\ref{sec:concentration}, we derive the exact enlargement theorem and the sharp two-set concentration inequality. In Section~\ref{sec:Hn-localization}, we apply the one-sided concentration estimate to Bamler's \(H_n\)-centers. In Section~\ref{sec:heat-kernel-estimates}, we prove convex-order and moment estimates for logarithmic derivatives of the conjugate heat kernel. In Section~\ref{sec:profile-regularity}, we record the entropy-regular profile and half-space stability consequences. In Section~\ref{sec:poincare}, we prove the Gaussian rearrangement inequality and recover the sharp log-Sobolev, Poincaré, Cheeger, and Faber--Krahn estimates. In Section~\ref{sec:reverse-hypercontractivity}, we prove reverse hypercontractivity. Finally, in Section~\ref{sec:path-space-bobkov}, we prove the path-space Bobkov inequality.

\subsection*{Acknowledgments}

I am grateful to Max Hallgren for mentioning that he has independently obtained a version of the finite-dimensional Gaussian isoperimetric inequality along a Ricci flow. I am thankful to Bennett Chow for constant support and encouragement, especially while I was a Ricci bud exploring the heat-kernel geometry of a Ricci flow-er. I am indebted to the late Dan Stroock for introducing me to many ideas in the theory of Gaussian measures.

\section{Bobkov inequality and Gaussian isoperimetry}
\label{sec:bobkov}

We prove the functional inequality behind Theorem~\ref{thm:intro-GI}. Our argument is inspired by the Bakry--Ledoux proof \cite{MR1374200} of Bobkov inequality, with one Ricci-flow feature: the Bochner Ricci term is exactly cancelled by the variation of the metric \eqref{eq:lovely-bochner}. We write
\begin{align*}
    P_{s,t_0}f(x_0)=\int_M f\,d\nu
\end{align*}
for the forward heat evolution from time \(s\) to time \(t_0\), evaluated at \(x_0\).

\begin{theorem}[Bobkov inequality]
\label{thm:Bobkov}
    For every smooth function \(f:M\to[0,1]\),
    \begin{align}
        I\left(\int_M f\,d\nu\right)
        \le
        \int_M
        \sqrt{I(f)^2+2\tau|\nabla f|_{g_s}^2}\,d\nu.
        \label{eq:Bobkov}
    \end{align}
    The constant is sharp in the static Euclidean flow.
\end{theorem}

\begin{proof}
    It suffices first to assume that \(\varepsilon\le f\le1-\varepsilon\). The general case follows by applying the result to \(f_\varepsilon=(1-2\varepsilon)f+\varepsilon\) and letting \(\varepsilon\downarrow0\).
    
    Let \(u(\cdot,t)=P_{s,t}f\), so that \(\Box u=0.\) Put \(a(t)=2(t_0-t),\)
    \begin{align*}
        Q=\sqrt{I(u)^2+a(t)|\nabla u|_{g_t}^2}.
    \end{align*}
    We prove that \(\Box Q\le0\).
    
    The Gaussian profile satisfies \(I'=-\Phi^{-1},\) \(I I''=-1,\) and hence \((I^2)''=2(I')^2-2.\) Since \(\Box u=0\), this gives
    \begin{align*}
        \Box I(u)^2
        =
        2(1-I'(u)^2)|\nabla u|^2.
    \end{align*}
    On the other hand, by the Ricci-flow Bochner cancellation \eqref{eq:lovely-bochner} and \(a'=-2\),
    \begin{align*}
        \Box\bigl(a|\nabla u|^2\bigr)
        =
        -2|\nabla u|^2-2a|\nabla^2u|^2.
    \end{align*}
    Therefore
    \begin{align}
        \Box Q^2
        =
        -2I'(u)^2|\nabla u|^2
        -
        2a|\nabla^2u|^2.
        \label{eq:BoxQsquare}
    \end{align}
    Also
    \begin{align*}
        \nabla Q^2
        =
        2I(u)I'(u)\nabla u
        +
        2a\nabla^2u(\nabla u,\cdot).
    \end{align*}
    Using
    \begin{align*}
        \Box Q
        =
        \frac{1}{2Q}\Box Q^2
        +
        \frac{1}{4Q^3}|\nabla Q^2|^2,
    \end{align*}
    and writing \(w=\Phi^{-1}(u),\) \(L=\sqrt{1+a|\nabla w|^2},\) \(Q=I(u)L,\) the preceding identities simplify to
    \begin{align}
            -\Box Q
            =
            \frac{aI(u)}{L}
            \left(
            |\nabla^2w|^2
            -
            \frac{a|\nabla^2w(\nabla w,\cdot)|^2}{L^2}
            \right).
            \label{eq:Bobkov-pointwise-defect}
    \end{align}
    The right-hand side is nonnegative since \(|\nabla^2w(\nabla w,\cdot)|^2 \le \nabla^2w|^2|\nabla w|^2\) and \(L^2=1+a|\nabla w|^2\). Hence \(\Box Q\le0\).
    
    Let \(\nu_t=\nu_{x_0,t_0;t}.\) For every smooth time-dependent function \(h\), \(\frac{d}{dt}\int_M h\,d\nu_t = \int_M \Box h\,d\nu_t.\) Thus \(t\longmapsto \int_M Q(\cdot,t)\,d\nu_t\) is nonincreasing. At \(t=t_0\), we have \(a(t_0)=0\) and \(\nu_{t_0}=\delta_{x_0}\), so
    \begin{align*}
        \int_M Q(\cdot,t_0)\,d\nu_{t_0}
        =
        I(P_{s,t_0}f(x_0)).
    \end{align*}
    At \(t=s\),
    \begin{align*}
        Q(\cdot,s)
        =
        \sqrt{I(f)^2+2\tau|\nabla f|_{g_s}^2}.
    \end{align*}
    The monotonicity gives \eqref{eq:Bobkov}.
\end{proof}

\begin{corollary}[Equality on a closed manifold]
\label{cor:Bobkov-equality}
    Suppose \(M\) is closed and connected. If \(0<f<1\) is smooth and equality holds in Theorem~\ref{thm:Bobkov}, then \(f\) is constant.
\end{corollary}

\begin{proof}
    Equality in Theorem~\ref{thm:Bobkov} implies that
    \begin{align*}
        \int_s^{t_0}\int_M -\Box Q\,d\nu_r\,dr=0.
    \end{align*}
    By the pointwise formula \eqref{eq:Bobkov-pointwise-defect}, the integrand is nonnegative. Since the conjugate heat-kernel density is positive for \(r<t_0\), the defect vanishes identically for \(r\in(s,t_0)\). In particular, with \(w_r=\Phi^{-1}(P_{s,r}f)\), we have
    \begin{align*}
        \nabla^2_{g_r}w_r=0
        \qquad
        \text{for every }r\in(s,t_0).
    \end{align*}
    Indeed, the term in parentheses in \eqref{eq:Bobkov-pointwise-defect} satisfies the stronger lower bound
    \begin{align*}
        |\nabla^2w|^2
        -
        \frac{a|\nabla^2w(\nabla w,\cdot)|^2}{L^2}
        \ge
        \frac{|\nabla^2w|^2}{L^2}.
    \end{align*}
    Thus vanishing of the defect forces \(\nabla^2w_r=0\).
    
    On a closed connected manifold, a function with zero Hessian has parallel gradient. Since the function attains its maximum and minimum, this parallel gradient must vanish. Hence \(w_r\), and therefore \(P_{s,r}f\), is spatially constant for every \(r\in(s,t_0)\). Letting \(r\downarrow s\) gives that \(f\) is constant.
\end{proof}

\begin{remark}
    In a complete bounded-geometry setting, provided the preceding monotonicity and integration by parts are justified, the same argument shows that a nonconstant equality case forces
    \begin{align*}
        \nabla^2_{g_r}\Phi^{-1}(P_{s,r}f)=0
        \qquad
        \text{for }r\in(s,t_0).
    \end{align*}
    Thus the gradient is parallel. If it is nonzero, the de Rham splitting theorem gives a static Euclidean line factor. Along this factor, \(\Phi^{-1}(P_{s,r}f)\) is affine, and passing to the initial time gives \(f(y,z)=\Phi(\alpha z+\beta)\) on the line factor. We shall not use this noncompact rigidity statement below.
\end{remark}

\begin{remark}
    The same proof, with \(a(t)\) replaced by \(\lambda+2(t_0-t)\), gives for every \(\lambda\ge0\)
    \begin{align}
        \sqrt{
        I(P_{s,t_0}f(x_0))^2
        +\lambda|\nabla P_{s,t_0}f|_{g_{t_0}}^2(x_0)
        }
        \le
        P_{s,t_0}\left[
        \sqrt{
        I(f)^2+
        \bigl(2(t_0-s)+\lambda\bigr)|\nabla f|_{g_s}^2
        }
        \right](x_0).
    \end{align}
    Letting \(\lambda\to\infty\), after an affine rescaling of a bounded smooth function, recovers the gradient contraction
    \begin{align}
        |\nabla P_{s,t_0}f|_{g_{t_0}}(x_0)
        \le
        P_{s,t_0}(|\nabla f|_{g_s})(x_0).
        \label{eq:gradient-contraction}
    \end{align}
\end{remark}

We now pass from the functional inequality to sets.

\begin{proof}[Proof of Theorem~\ref{thm:intro-GI}]
    If \(\operatorname{Per}_\nu(E)=+\infty\), there is nothing to prove. Otherwise, by the relaxation definition of the weighted perimeter, choose smooth functions \(0\le f_j\le1\) such that \(f_j\to\mathbf 1_E\) in \(L^1(\nu)\) and
    \begin{align*}
        \liminf_{j\to\infty}\int_M|\nabla f_j|\,d\nu
        \le
        \operatorname{Per}_\nu(E)+\varepsilon.
    \end{align*}
    Applying Theorem~\ref{thm:Bobkov} to \(f_j\), and using \(\sqrt{a^2+b^2}\le a+b\), gives
    \begin{align*}
        I\left(\int_M f_j\,d\nu\right)
        \le
        \int_M I(f_j)\,d\nu
        +
        \sqrt{2\tau}\int_M|\nabla f_j|\,d\nu.
    \end{align*}
    Since \(I(0)=I(1)=0\), \(I\) is bounded on \([0,1]\), and \(f_j\to\mathbf 1_E\) in measure, we have \(\int_M I(f_j)\,d\nu\to0.\) Also \(\int_M f_j\,d\nu\to \nu(E).\) Letting \(j\to\infty\), we get
    \begin{align*}
        I(\nu(E))
        \le
        \sqrt{2\tau}\bigl(\operatorname{Per}_\nu(E)+\varepsilon\bigr).
    \end{align*}
    Finally let \(\varepsilon\downarrow0\). This proves
    \begin{align*}
        \operatorname{Per}_{\nu}(E)
        \ge
        \frac1{\sqrt{2\tau}}I(\nu(E)).
    \end{align*}
    Taking the infimum over all sets of fixed \(\nu\)-measure gives the profile form.
\end{proof}

The isoperimetric inequality is equivalently an exact enlargement estimate.

\begin{corollary}[Exact Gaussian enlargement]
\label{cor:enlargement}
    For every Borel set \(E\subset M\) and every \(r\ge0\),
    \begin{align}
        \nu(E_r)
        \ge
        \Phi\left(
        \Phi^{-1}(\nu(E))+\frac{r}{\sqrt{2\tau}}
        \right),
        \label{eq:exact-enlargement}
    \end{align}
    where \(E_0=E\) and, for \(r>0\), \(E_r=\{y:d_{g_s}(y,E)<r\}.\)
\end{corollary}

\begin{proof}
    The endpoint cases \(\nu(E)=0\) and \(\nu(E)=1\) are immediate, with the usual extended interpretation of \(\Phi^{-1}\). Assume \(0<\nu(E)<1\).
    
    For \(r>0\), set \(v(r)=\nu(E_r).\) The function \(v\) is monotone. By the standard Minkowski-content form of perimeter for tubular neighborhoods, for a.e. \(r>0\), \(v'(r)\ge \operatorname{Per}_\nu(E_r).\) Applying Theorem~\ref{thm:intro-GI} to \(E_r\) gives, for a.e. \(r>0\), \(v'(r)\ge\frac1{\sqrt{2\tau}}I(v(r)).\) On the set where \(0<v(r)<1\), we may compose with \(\Phi^{-1}\), and since \((\Phi^{-1})'=\frac1I,\) we obtain
    \begin{align*}
        \frac{d}{dr}\Phi^{-1}(v(r))
        \ge
        \frac1{\sqrt{2\tau}}
    \end{align*}
    for a.e. \(r\). Integrating from \(\rho\) to \(r\), with \(0<\rho<r\), gives
    \begin{align*}
        \Phi^{-1}(v(r))
        \ge
        \Phi^{-1}(v(\rho))
        +
        \frac{r-\rho}{\sqrt{2\tau}}.
    \end{align*}
    Letting \(\rho\downarrow0\), and using \(\liminf_{\rho\downarrow0}v(\rho)\ge \nu(E),\) we get
    \begin{align*}
        \Phi^{-1}(v(r))
        \ge
        \Phi^{-1}(\nu(E))
        +
        \frac{r}{\sqrt{2\tau}}.
    \end{align*}
    Applying \(\Phi\) gives \eqref{eq:exact-enlargement}. The case \(r=0\) is the definition \(E_0=E\).
\end{proof}

\section{Concentration phenomena}
\label{sec:concentration}

We write \(d_s\) for the distance induced by \(g_s\), and \(d_s(A,B)=\inf\{d_s(a,b):a\in A,\ b\in B\}.\) The exact enlargement estimate gives a Gaussian-quantile version of the Hein--Naber two-set concentration theorem \cite[Theorem~1.13]{MR3245102}. The exponential concentration estimate is then recovered by applying the standard Gaussian tail bound.

\begin{theorem}[Two-set concentration]
\label{thm:two-set-concentration}
    Let \(A,B\subset M\) be Borel sets with \(d_s(A,B)>0\). Then
    \begin{align}
        \Phi^{-1}(\nu(A))+\Phi^{-1}(\nu(B))
        \le
        -\frac{d_s(A,B)}{\sqrt{2\tau}}.
        \label{eq:two-set-quantile}
    \end{align}
    Consequently,
    \begin{align}
        \nu(A)\nu(B)
        \le
        \Phi\left(-\frac{d_s(A,B)}{2\sqrt{2\tau}}\right)^2
        \le
        \exp\left(-\frac{d_s(A,B)^2}{8\tau}\right).
        \label{eq:two-set-product}
    \end{align}
\end{theorem}

\begin{proof}
    Put \(d=d_s(A,B)\). Applying Corollary~\ref{cor:enlargement} to \(B\), we obtain, for every \(0<r<d\),
    \begin{align*}
        \nu(B_r)
        \ge
        \Phi\left(
        \Phi^{-1}(\nu(B))+\frac{r}{\sqrt{2\tau}}
        \right).
    \end{align*}
    Since \(A\cap B_r=\emptyset\),
    \begin{align*}
        \nu(A)
        \le
        1-\Phi\left(
        \Phi^{-1}(\nu(B))+\frac{r}{\sqrt{2\tau}}
        \right).
    \end{align*}
    Letting \(r\uparrow d\) gives \eqref{eq:two-set-quantile}. The endpoint cases \(\nu(A),\nu(B)\in\{0,1\}\) follow by the same approximation, with the usual extended interpretation of \(\Phi^{-1}\). 
    
    Set \(D=\frac{d}{\sqrt{2\tau}},\) \(x=\Phi^{-1}(\nu(A)),\) \(y=\Phi^{-1}(\nu(B)).\) Under the constraint \(x+y\le -D\), the product \(\Phi(x)\Phi(y)\) is maximized when \(x=y=-D/2\). Indeed, it is enough to consider \(y=-D-x\), and the function \(h(x)=\log\Phi(x)+\log\Phi(-D-x)\) is concave and symmetric about \(-D/2\), since \(\Phi\) is log-concave. Hence \(h\) attains its maximum at \(x=-D/2\). This proves the first inequality in \eqref{eq:two-set-product}. The second follows from \(\Phi(-u)\le e^{-u^2/2}\) for \(u\ge0\).
\end{proof}

The same argument gives a one-sided form. This is often the more useful version: a set separated from a set of definite heat-kernel mass has the corresponding Gaussian tail.

\begin{corollary}[One-sided concentration]
\label{cor:one-sided-concentration}
    Let \(A,B\subset M\) be Borel sets and suppose \(\nu(B)\ge\beta>0\). If \(d_s(A,B)>0\), then
    \begin{align}
        \nu(A)
        \le
        1-\Phi\left(
        \Phi^{-1}(\beta)+\frac{d_s(A,B)}{\sqrt{2\tau}}
        \right).
        \label{eq:one-sided-concentration}
    \end{align}
\end{corollary}

\begin{proof}
    This is the estimate in the proof of Theorem~\ref{thm:two-set-concentration}, with \(\nu(B)\) replaced by the lower bound \(\beta\).
\end{proof}

The enlargement estimate also controls quantiles of Lipschitz observables. For a measurable function \(F\), let
\begin{align}
    q_a(F)=\inf\{r:\nu(F\le r)\ge a\}
\end{align}
be its \textit{\(a\)-quantile.}

\begin{corollary}[Lipschitz quantiles]
\label{cor:lipschitz-quantiles}
    Let \(F:(M,g_s)\to\mathbb R\) be \(L\)-Lipschitz. Then, for every \(0<a<b<1\),
    \begin{align}
        q_b(F)-q_a(F)
        \le
        L\sqrt{2\tau}\,
        \bigl(\Phi^{-1}(b)-\Phi^{-1}(a)\bigr).
        \label{eq:lipschitz-quantiles}
    \end{align}
    In particular, if \(m\) is a median of \(F\), then, for every \(r>0\),
    \begin{align}
        \nu(F\ge m+r)
        \le
        1-\Phi\left(\frac{r}{L\sqrt{2\tau}}\right),
        \qquad
        \nu(F\le m-r)
        \le
        1-\Phi\left(\frac{r}{L\sqrt{2\tau}}\right),
    \end{align}
    with the evident interpretation when \(L=0\).
\end{corollary}

\begin{proof}
    By the usual approximation of quantiles, it is enough to argue with \(E=\{F\le q_a(F)+\varepsilon\}\) and then let \(\varepsilon\downarrow0\). We have \(\nu(E)\ge a\). Since \(F\) is \(L\)-Lipschitz, \(E_\rho\subset \{F\le q_a(F)+\varepsilon+L\rho\}.\) By Corollary~\ref{cor:enlargement},
    \begin{align*}
        \nu(E_\rho)
        \ge
        \Phi\left(
        \Phi^{-1}(a)+\frac{\rho}{\sqrt{2\tau}}
        \right).
    \end{align*}
    Taking \(\rho = \sqrt{2\tau}\bigl(\Phi^{-1}(b)-\Phi^{-1}(a)\bigr),\) we get \(\nu(E_\rho)\ge b\). Hence
    \begin{align*}
        q_b(F)\le q_a(F)+\varepsilon+L\rho.
    \end{align*}
    Letting \(\varepsilon\downarrow0\) proves \eqref{eq:lipschitz-quantiles}.
    
    For the median tail, apply the enlargement estimate directly to \(\{F\le m\}\), which has measure at least \(1/2\), and then to \(\{-F\le -m\}\).
\end{proof}

\begin{remark}[Exponential integrability]
    The median tail bounds imply the corresponding Gaussian exponential integrability. If \(m\) is a median of an \(L\)-Lipschitz function \(F\), then
    \begin{align*}
        \nu(|F-m|\ge r)
        \le
        2\Phi\left(-\frac{r}{L\sqrt{2\tau}}\right),
        \qquad r\ge0,
    \end{align*}
    with the evident interpretation when \(L=0\). Thus every increasing function of \(|F-m|\) is dominated by the corresponding function of \(|G|\), where \(G\sim N(0,2\tau L^2)\). Consequently, for every \(\lambda\ge0\),
    \begin{align*}
        \int_M e^{\lambda |F-m|}\,d\nu
        \le
        2e^{\tau L^2\lambda^2}
        \Phi(\lambda L\sqrt{2\tau}),
    \end{align*}
    and, for \(0\le\beta<1/4\),
    \begin{align*}
        \int_M
        \exp\left(\beta\frac{|F-m|^2}{\tau L^2}\right)d\nu
        \le
        (1-4\beta)^{-1/2}.
    \end{align*}
\end{remark}

\begin{remark}
    The exponential estimate in \eqref{eq:two-set-product} is the Hein--Naber concentration bound \cite[Theorem~1.13]{MR3245102}. The quantile inequality \eqref{eq:two-set-quantile} is the corresponding Gaussian-profile form, and is sharp in the static Euclidean flow. Equivalently, every \(1\)-Lipschitz observable has \textit{width} at most \(2\sqrt{2\tau}\,\Phi^{-1}\left(1-\frac{\delta}{2}\right)\) on a set of \(\nu\)-measure at least \(1-\delta\).
\end{remark}

\section{Localization near Bamler's \texorpdfstring{\(H_n\)}{H}-centers}
\label{sec:Hn-localization}

We keep the notation from the previous sections: \(d\nu=K(x_0,t_0;\cdot,s)\,dg_s,\) \(\tau=t_0-s.\) We write \(d_s\) and \(B_s\) for the distance and balls of \(g_s\). Following Bamler \cite[Definition~3.10]{bamler2020entropy}, we call a point \(z\in M\) an \(H_n\)-center for \((x_0,t_0)\) at time \(s\) if
\begin{align}
    \int_M d_s(z,y)^2\,d\nu(y)\le H_n\tau,
    \qquad
    H_n=4+\frac{(n-1)\pi^2}{2}.
    \label{eq:Hn-center}
\end{align}
Thus, by Chebyshev inequality gives the rough localization estimate,
\begin{align*}
    \nu\left(B_s(z,\sqrt{A H_n\tau})\right)
    \ge
    1-\frac1A
    \qquad
    \text{for every }A>1.
\end{align*}
The point \(z\) is a probabilistic substitute, at the earlier time \(s\), for the spacetime basepoint \((x_0,t_0)\). The one-sided concentration estimate from the previous section upgrades the preceding \(L^2\)-localization into Gaussian-profile tail bounds. This gives a quantitative version of localization near \(H_n\)-centers in the same spirit as Bamler's heat-kernel estimates \cite[Theorem 3.14]{bamler2020entropy}.

\begin{theorem}[Tails from an \(H_n\)-center]
\label{thm:Hn-tail}
    Let \(z\) satisfy \eqref{eq:Hn-center}. Then, for every \(A>1\) and every \(r\ge\sqrt{A H_n\tau}\),
    \begin{align}
        \nu(M\setminus B_s(z,r))
        \le
        1-\Phi\left(
        \Phi^{-1}\left(1-\frac1A\right)
        +
        \frac{r-\sqrt{A H_n\tau}}{\sqrt{2\tau}}
        \right).
        \label{eq:Hn-profile-tail}
    \end{align}
    In particular,
    \begin{align}
        \nu(M\setminus B_s(z,r))
        \le
        \Phi\left(
        -\frac{r-\sqrt{2H_n\tau}}{\sqrt{2\tau}}
        \right),
        \qquad
        r\ge \sqrt{2H_n\tau}.
        \label{eq:Hn-median-tail}
    \end{align}
    Consequently, for every \(r\ge0\),
    \begin{align}
        \nu(M\setminus B_s(z,r))
        \le
        \exp\left(
        -\frac{(r-\sqrt{2H_n\tau})_+^2}{4\tau}
        \right).
        \label{eq:Hn-exp-tail}
    \end{align}
\end{theorem}

\begin{proof}
    Put \(R_A=\sqrt{A H_n\tau}.\) By \eqref{eq:Hn-center}, \(\nu(B_s(z,R_A))\ge 1-\frac1A.\) If \(r>R_A\), then
    \begin{align*}
        d_s(M\setminus B_s(z,r),B_s(z,R_A))\ge r-R_A
    \end{align*}
    up to the usual open--closed approximation of balls. Applying Corollary~\ref{cor:one-sided-concentration} with \(\beta=1-1/A\) gives \eqref{eq:Hn-profile-tail}. The endpoint \(r=R_A\) follows directly from Chebyshev inequality. Taking \(A=2\) gives \eqref{eq:Hn-median-tail}, since \(\Phi^{-1}(1/2)=0\). Finally, for
    \begin{align*}
        u=\frac{r-\sqrt{2H_n\tau}}{\sqrt{2\tau}}\ge0,
    \end{align*}
    the bound \(\Phi(-u)\le e^{-u^2/2}\) gives \eqref{eq:Hn-exp-tail}. When \(r<\sqrt{2H_n\tau}\), the estimate is trivial.
\end{proof}

The same tail estimate can be expressed in terms of quantiles of the distance function \(y\mapsto d_s(z,y)\). We record the median-based form. Let \(q_b\) be the \(b\)-quantile of \(y\mapsto d_s(z,y)\), namely
\begin{align*}
    q_b=\inf\{r:\nu(B_s(z,r))\ge b\}.
\end{align*}

\begin{corollary}[Quantile localization]
\label{cor:Hn-quantile}
    For every \(1/2\le b<1\),
    \begin{align}
        q_b
        \le
        \sqrt{2H_n\tau}
        +
        \sqrt{2\tau}\,\Phi^{-1}(b).
        \label{eq:Hn-quantile}
    \end{align}
\end{corollary}

\begin{proof}
    For \(b=1/2\), the claim follows from \eqref{eq:Hn-center} and Chebyshev inequality. If \(b>1/2\), set \(r=\sqrt{2H_n\tau}+\sqrt{2\tau}\,\Phi^{-1}(b).\) Then \eqref{eq:Hn-median-tail} gives
    \begin{align*}
        \nu(M\setminus B_s(z,r))
        \le
        \Phi(-\Phi^{-1}(b))
        =
        1-b.
    \end{align*}
    Thus \(\nu(B_s(z,r))\ge b\), which is \eqref{eq:Hn-quantile}.
\end{proof}

The tail comparison also yields half-Gaussian moment bounds for the excess distance beyond the radius \(\sqrt{2H_n\tau}\).

\begin{corollary}[Exponential moments of the excess distance]
\label{cor:Hn-excess-moments}
    Set \(Y(y)=\bigl(d_s(z,y)-\sqrt{2H_n\tau}\bigr)_+.\) Then, for every \(\lambda\ge0\),
    \begin{align}
        \int_M e^{\lambda Y}\,d\nu
        \le
        \frac12
        +
        e^{\tau\lambda^2}\Phi(\lambda\sqrt{2\tau}).
        \label{eq:Hn-exp-moment}
    \end{align}
    Equivalently,
    \begin{align}
        \int_M e^{\lambda d_s(z,y)}\,d\nu(y)
        \le
        e^{\lambda\sqrt{2H_n\tau}}
        \left[
        \frac12
        +
        e^{\tau\lambda^2}\Phi(\lambda\sqrt{2\tau})
        \right].
        \label{eq:Hn-distance-exp-moment}
    \end{align}
    Moreover, for every \(0\le\beta<1/4\),
    \begin{align}
        \int_M
        \exp\left(\beta\frac{Y^2}{\tau}\right)d\nu
        \le
        \frac12+\frac12(1-4\beta)^{-1/2}.
        \label{eq:Hn-square-exp-moment}
    \end{align}
\end{corollary}

\begin{proof}
    Let \(G\sim N(0,2\tau)\) and \(G_+=\max\{G,0\}\). By \eqref{eq:Hn-median-tail}, for every \(r>0\),
    \begin{align*}
        \nu(Y>r)
        \le
        \Phi\left(-\frac{r}{\sqrt{2\tau}}\right)
        =
        \mathbb P(G_+>r).
    \end{align*}
    Thus every increasing function of \(Y\) is dominated by that of \(G_+\). Applying this domination to the increasing functions \(e^{\lambda r}\) and \(\exp(\beta r^2/\tau)\) gives the desired estimates after the elementary half-Gaussian computations
    \begin{align*}
        \mathbb E e^{\lambda G_+}
        =
        \frac12+e^{\tau\lambda^2}\Phi(\lambda\sqrt{2\tau}),\qquad \mathbb E\exp\left(\beta\frac{G_+^2}{\tau}\right)
        =
        \frac12+\frac12(1-4\beta)^{-1/2}.
    \end{align*}
\end{proof}

\section{Integral estimates for logarithmic derivatives of the heat-kernel}
\label{sec:heat-kernel-estimates}

We next record two consequences for the conjugate heat kernel. The first is an averaged upper bound for the heat-kernel mass of sets separated from a region carrying definite heat-kernel mass. The second concerns logarithmic derivatives of the heat kernel. There we combine Bamler's integral logarithmic-derivative estimate \eqref{eq:score-set-bound} with a rearrangement argument to obtain convex-order and moment bounds with the universal Gaussian model constants.

Throughout this section we fix \(x\in M\), \(s<t\), put \(\tau=t-s\), and write \(d\nu(y)=K(x,t;y,s)\,dg_s(y).\) All distances and volumes in this section are taken with respect to \(g_s\).

\subsection{Average upper bounds}

The one-sided concentration estimate converts separation from a set of definite heat-kernel mass into a bound for the heat-kernel mass of the separated set. This is the same concentration mechanism that appears in the average estimates of Hein--Naber; the form below keeps the full Gaussian quantile profile. Compare \cite[Corollary~1.14]{MR3245102}.

\begin{proposition}[Average heat-kernel upper bound]
\label{prop:average-heat-kernel-upper}
    Let \(A,B\subset M\) be Borel sets and assume \(\nu(B)\ge\beta>0\). If \(D=d_s(A,B)>0\), then
    \begin{align}
        \int_A K(x,t;y,s)\,dg_s(y)
        \le
        1-\Phi\left(
        \Phi^{-1}(\beta)+\frac{D}{\sqrt{2\tau}}
        \right).
        \label{eq:average-HK-mass}
    \end{align}
    Consequently, if \(|A|_{g_s}\ge v_0\tau^{n/2}\), then
    \begin{align}
        \frac{1}{|A|_{g_s}}
        \int_A K(x,t;y,s)\,dg_s(y)
        \le
        \frac{1}{v_0\tau^{n/2}}
        \left[
        1-\Phi\left(
        \Phi^{-1}(\beta)+\frac{D}{\sqrt{2\tau}}
        \right)
        \right].
        \label{eq:average-HK-upper}
    \end{align}
\end{proposition}

\begin{proof}
    Since \(\nu(A)=\int_A K(x,t;y,s)\,dg_s(y)\), the first estimate is exactly Corollary~\ref{cor:one-sided-concentration}. Dividing by the lower bound for \(|A|_{g_s}\) gives \eqref{eq:average-HK-upper}.
\end{proof}

For fixed \(\beta\), the right-hand side has the Gaussian envelope
\begin{align*}
    1-\Phi\left(
    \Phi^{-1}(\beta)+\frac{D}{\sqrt{2\tau}}
    \right)
    \le
    C(\beta)
    \exp\left(
    -\frac{(D-C(\beta)\sqrt{\tau})_+^2}{4\tau}
    \right).
\end{align*}
Thus the quantile estimate recovers the usual Gaussian off-diagonal decay at the averaged level. Obtaining pointwise heat-kernel bounds requires additional analytic input, and we do not pursue that here.

\subsection{Logarithmic derivatives of the heat kernel}

We next estimate logarithmic derivatives of the heat kernel in the basepoint variable, refining Bamler's integral estimate \cite[Proposition~4.2]{bamler2020entropy}.

Fix a unit vector \(v\in T_xM\). Define the logarithmic derivative
\begin{align}
    S_v(y)
    =
    \left\langle
    \nabla_x\log K(x,t;y,s),v
    \right\rangle_{g_t}.
\end{align}
Equivalently, \(S_v(y)K(x,t;y,s)=\partial_v^xK(x,t;y,s).\) By parabolic rescaling and Bamler's equation \cite[(4.5)]{bamler2020entropy}, for every Borel set \(X\subset M\),
\begin{align}
    \left|
    \int_X S_v\,d\nu
    \right|
    \le
    \frac{1}{\sqrt{2\tau}}I(\nu(X)).
    \label{eq:score-set-bound}
\end{align}
Indeed, Bamler's one-dimensional profile is normalized with variance \(2\). In the notation of this paper, his function \(\Phi'(\Phi^{-1}(a))\) is \(I(a)/\sqrt2\). Applying his one-sided estimate to \(v\) and to \(-v\), and then rescaling from time \(1\) to time \(\tau=t-s\), gives \eqref{eq:score-set-bound}.

The following result says that the logarithmic derivative is dominated, in convex order, by the corresponding Euclidean Gaussian model.

\begin{theorem}[Convex order for logarithmic derivatives]
\label{thm:score-convex-order}
    Let \(Z_\tau\sim N(0,1/(2\tau))\). Then for every convex function \(\psi:\mathbb R\to\mathbb R\) for which both sides are well-defined,
    \begin{align}
        \int_M \psi(S_v)\,d\nu
        \le
        \mathbb E\psi(Z_\tau).
        \label{eq:score-convex-order}
    \end{align}
\end{theorem}

\begin{proof}
    First,
    \begin{align*}
        \int_M S_v\,d\nu
        =
        \partial_v^x\int_M K(x,t;y,s)\,dg_s(y)
        =
        \partial_v^x 1
        =
        0.
    \end{align*}
    Thus \(S_v\) and \(Z_\tau\) have the same mean.
    
    Let \(Y^*\) denote the decreasing rearrangement of a real-valued random variable \(Y\):
    \begin{align*}
        Y^*(a)
        \coloneqq
        \inf\{\lambda:\mathbb P(Y>\lambda)\le a\}.
    \end{align*}
    Since the measure \(\nu\) is non-atomic, the Hardy--Littlewood maximal rearrangement identity gives, for \(0<a<1\),
    \begin{align*}
        \int_0^a S_v^*(r)\,dr
        =
        \sup_{\nu(X)=a}\int_X S_v\,d\nu.
    \end{align*}
    By \eqref{eq:score-set-bound},
    \begin{align}
        \int_0^a S_v^*(r)\,dr
        \le
        \frac1{\sqrt{2\tau}}I(a).
        \label{eq:score-rearrangement}
    \end{align}
    For \(Z_\tau\), one has \(Z_\tau^*(r) = \frac1{\sqrt{2\tau}}\Phi^{-1}(1-r),\) and hence
    \begin{align}
        \int_0^a Z_\tau^*(r)\,dr
        =
        \frac1{\sqrt{2\tau}}I(a).
        \label{eq:Gaussian-score-rearrangement}
    \end{align}
    Indeed,
    \begin{align*}
        \int_0^a\Phi^{-1}(1-r)\,dr
        =
        \int_{\Phi^{-1}(1-a)}^\infty z\varphi(z)\,dz
        =
        \varphi(\Phi^{-1}(1-a))
        =
        I(a).
    \end{align*}
    Combining \eqref{eq:score-rearrangement} and \eqref{eq:Gaussian-score-rearrangement}, and using equality of the means, the Hardy--Littlewood--Pólya characterization of convex order proves \eqref{eq:score-convex-order}; see \cite{MR46395} or \cite[Theorem~3.A.5]{MR2265633}.
\end{proof}

The moment estimates now follow from the Gaussian model.

\begin{corollary}[Moment estimates for logarithmic derivatives]
\label{cor:score-moments}
    For every \(p\ge1\),
    \begin{align}
        \tau^{p/2}\int_M |S_v|^p\,d\nu
        \le
        \frac{\Gamma((p+1)/2)}{\sqrt{\pi}}.
        \label{eq:directional-score-moment}
    \end{align}
    Consequently,
    \begin{align}
        \tau^{p/2}
        \int_M
        |\nabla_x\log K(x,t;y,s)|_{g_t}^p\,d\nu(y)
        \le
        \frac{\Gamma((n+p)/2)}{\Gamma(n/2)}.
        \label{eq:vector-score-moment}
    \end{align}
    Equivalently,
    \begin{align}
        \tau^{p/2}
        \int_M
        \frac{|\nabla_xK(x,t;y,s)|_{g_t}^p}
        {K(x,t;y,s)^{p-1}}
        \,dg_s(y)
        \le
        \frac{\Gamma((n+p)/2)}{\Gamma(n/2)}.
        \label{eq:integral-HK-gradient}
    \end{align}
\end{corollary}

\begin{proof}
    Apply Theorem~\ref{thm:score-convex-order} to \(\psi(r)=|r|^p\). Since \(Z_\tau\sim N(0,1/(2\tau))\),
    \begin{align*}
        \mathbb E|Z_\tau|^p
        =
        \tau^{-p/2}
        \frac{\Gamma((p+1)/2)}{\sqrt{\pi}},
    \end{align*}
    which proves \eqref{eq:directional-score-moment}.
    
    For the vector estimate, write \(S(y)=\nabla_x\log K(x,t;y,s).\) Averaging the directional estimate over the unit sphere in \(T_xM\), with normalized surface measure, and using
    \begin{align*}
        \int_{S^{n-1}}|\langle S,\theta\rangle|^p\,d\theta
        =
        \frac{\Gamma(n/2)\Gamma((p+1)/2)}
        {\sqrt{\pi}\Gamma((n+p)/2)}
        |S|^p,
    \end{align*}
    gives \eqref{eq:vector-score-moment}. Finally, \eqref{eq:integral-HK-gradient} is the same estimate written with \(d\nu=K\,dg_s\).
\end{proof}

The same convex-order comparison gives localized bounds on subsets of small heat-kernel measure. We record the directional form, which is useful when the heat-kernel gradient is integrated over a prescribed region.

\begin{corollary}[Localized directional moments]
\label{cor:localized-score-moments}
    Let \(X\subset M\) be Borel, put \(a=\nu(X)\), and, if \(a>0\), set \(b_a=\Phi^{-1}\left(1-\frac a2\right).\) Then, for every \(p\ge1\),
    \begin{align}
        \int_X |S_v|^p\,d\nu
        \le
        (2\tau)^{-p/2}
        2\int_{b_a}^{\infty} r^p\varphi(r)\,dr,
        \label{eq:localized-directional-score}
    \end{align}
    with the convention that the right-hand side is \(0\) when \(a=0\). In particular,
    \begin{align*}
        \tau^{p/2}\int_X |S_v|^p\,d\nu
        \le
        C(p)\,\nu(X)
        \left(\log\frac{2}{\nu(X)}\right)^{p/2}
    \end{align*}
    for \(0<\nu(X)\le1\).
\end{corollary}

\begin{proof}
    The case \(a=0\) is trivial, so assume \(0<a\le1\). For \(c\ge0\), the function \(r\longmapsto (|r|^p-c)_+\) is convex. Hence Theorem~\ref{thm:score-convex-order} gives
    \begin{align*}
        \int_M (|S_v|^p-c)_+\,d\nu
        \le
        \mathbb E(|Z_\tau|^p-c)_+.
    \end{align*}
    Therefore
    \begin{align*}
        \int_X |S_v|^p\,d\nu
        \le
        ca+\mathbb E(|Z_\tau|^p-c)_+.
    \end{align*}
    Taking the infimum over \(c\ge0\), and using the Hardy--Littlewood maximal rearrangement formula for nonnegative random variables, gives
    \begin{align*}
        \int_X |S_v|^p\,d\nu
        \le
        \int_0^a (|Z_\tau|^p)^*(r)\,dr.
    \end{align*}
    Write \(Z_\tau=(2\tau)^{-1/2}G,\) \(G\sim N(0,1).\) The level \(b_a\) is chosen so that \(\mathbb P(|G|\ge b_a)=a,\) or equivalently \(b_a=\Phi^{-1}\left(1-\frac a2\right).\) Thus
    \begin{align*}
        \int_0^a (|Z_\tau|^p)^*(r)\,dr
        =
        \mathbb E\left[
        |Z_\tau|^p\mathbf 1_{\{|G|\ge b_a\}}
        \right]
        =
        (2\tau)^{-p/2}
        2\int_{b_a}^{\infty}r^p\varphi(r)\,dr.
    \end{align*}
    This proves \eqref{eq:localized-directional-score}. The logarithmic estimate follows from the standard Gaussian tail bound
    \begin{align*}
        \int_b^\infty r^p\varphi(r)\,dr
        \le
        C(p)(1-\Phi(b))(1+b^p),
        \qquad b\ge0,
    \end{align*}
    together with \(1-\Phi(b_a)=a/2\) and \(1+b_a^p\le C(p)\left(\log\frac2a\right)^{p/2}.\)
\end{proof}

\begin{remark}
    In the static Euclidean flow, \(S_v\) has law \(N(0,1/(2\tau))\). Thus the convex-order theorem and the directional moment constants are sharp as universal constants.
\end{remark}

\section{Gaussian-profile regularity}
\label{sec:profile-regularity}

We next combine the sharp isoperimetric profile with Bamler's \(\varepsilon\)-regularity theorem \cite[Theorems~10.2 and~10.3]{bamler2020entropy}. The lower bound from Theorem~\ref{thm:intro-GI} holds for every conjugate heat-kernel measure, with no regularity assumption. When the pointed Nash entropy is close to its Euclidean value, Bamler's theorem gives regularity charts in which the conjugate heat-kernel measure is close to the Euclidean Gaussian. In such charts, Euclidean half-spaces provide almost optimal competitors. This gives an almost sharp upper bound for the profile, and a corresponding stability statement for almost extremizing sets.

Let \((x,t)\in M\times I\), let \(r>0\), and assume \([t-r^2,t]\subset I\). For \(0<\theta<1\), set \(\nu_\theta=\nu_{x,t;t-\theta r^2}.\) We write \(\mathcal I_{\nu_\theta}\) for the weighted isoperimetric profile of \((M,g_{t-\theta r^2},\nu_\theta)\). Since the time difference is \(\theta r^2\), Theorem~\ref{thm:intro-GI} gives the universal lower bound
\begin{align}
    \mathcal I_{\nu_\theta}(a)
    \ge
    \frac{1}{\sqrt{2\theta}\,r}I(a),
    \qquad 0\le a\le1.
    \label{eq:universal-profile-lower}
\end{align}
Set \(K(x,t;y,t-r^2)=(4\pi r^2)^{-n/2}e^{-f(y)}.\) We define the pointed Nash entropy by
\begin{align*}
    \CMcal N_{x,t}(r^2)
    =
    \int_M f\,d\nu_{x,t;t-r^2}-\frac n2.
\end{align*}

We shall use the following compactness statement. The smooth Euclidean convergence is the standard compactness consequence of Bamler's \(\varepsilon\)-regularity theorem. We include the tightness conclusion because the stability argument requires convergence of the heat kernel mass together with the local smooth convergence on compact subsets.

\begin{proposition}
\label{prop:entropy-regular-Gaussian-convergence}
    Let \((M_i^n,g^i_\tau)_{\tau\in[-1,0]}\) be compact Ricci flows and let \(x_i\in M_i\). Assume \(\CMcal N^{(i)}_{x_i,0}(1)\to0.\) Fix \(0<\theta<1\), and set \(\nu_i=\nu^{(i)}_{x_i,0;-\theta}.\) Then, after passing to a subsequence, the pointed flows converge smoothly on compact subsets to the static Euclidean flow \((\mathbb R^n,g_{\mathrm{Euc}},0).\) More precisely, after a diagonal choice of convergence maps, for every \(R<\infty\) there are smooth embeddings \(\Psi_i:B_{\mathbb R^n}(0,R)\longrightarrow M_i,\) \(\Psi_i(0)=x_i,\) such that
    \begin{align*}
        \Psi_i^*g^i_{-\theta}\to g_{\mathrm{Euc}}
        \qquad
        \text{smoothly on }B(0,R),
    \end{align*}
    and
    \begin{align*}
        \Psi_i^*\nu_i
        \longrightarrow
        \gamma_\theta
        =
        (4\pi\theta)^{-n/2}e^{-|y|^2/4\theta}\,dy
    \end{align*}
    locally smoothly as measures. Moreover, the convergence is tight:
    \begin{align}
        \lim_{R\to\infty}\limsup_{i\to\infty}
        \nu_i\bigl(M_i\setminus\Psi_i(B(0,R))\bigr)
        =
        0.
        \label{eq:entropy-regular-tightness}
    \end{align}
\end{proposition}

\begin{proof}
    The smooth pointed convergence to the static Euclidean flow is the standard compactness consequence of Bamler's \(\varepsilon\)-regularity theorem, in the large-neighborhood form \cite[Theorems~10.2 and~10.3]{bamler2020entropy}. Indeed, when the pointed Nash entropy tends to its Euclidean value, Bamler's regularity theorem gives curvature control on larger and larger parabolic neighborhoods, and the compactness argument in its proof identifies the smooth pointed limit with flat Euclidean space.
    
    The conjugate heat kernels satisfy the conjugate heat equation. Under the smooth local convergence of the flows, parabolic compactness gives local smooth convergence of the heat-kernel densities. Since the limiting flow is the static Euclidean flow based at \((0,0)\), the limiting measure at time \(-\theta\) is \(d\gamma_\theta(y) = (4\pi\theta)^{-n/2}e^{-|y|^2/4\theta}\,dy.\)
    
    It remains to record tightness. Let \(z_i\) be an \(H_n\)-center for \((x_i,0)\) at time \(-\theta\). Then
    \begin{align*}
        \int_{M_i}d_{g^i_{-\theta}}(z_i,y)^2\,d\nu_i(y)
        \le
        H_n\theta.
    \end{align*}
    By the local convergence to the Euclidean heat kernel, there are constants \(m_0>0\) and \(C_0<\infty\), independent of \(i\), such that \(\nu_i\bigl(B_{g^i_{-\theta}}(x_i,C_0)\bigr)\ge m_0\) for all large \(i\). Hence
    \begin{align*}
        H_n\theta
        \ge
        m_0\bigl(d_{g^i_{-\theta}}(x_i,z_i)-C_0\bigr)_+^2.
    \end{align*}
    Thus the \(H_n\)-centers \(z_i\) remain at uniformly bounded \(g^i_{-\theta}\)-distance from \(x_i\).
    
    Now apply Theorem~\ref{thm:Hn-tail} at time \(-\theta\). For \(\rho\ge\sqrt{2H_n\theta}\),
    \begin{align*}
        \nu_i\bigl(M_i\setminus B_{g^i_{-\theta}}(z_i,\rho)\bigr)
        \le
        \Phi\left(
        -\frac{\rho-\sqrt{2H_n\theta}}{\sqrt{2\theta}}
        \right).
    \end{align*}
    Since \(z_i\) is uniformly close to \(x_i\), this gives
    \begin{align*}
        \lim_{R\to\infty}\limsup_{i\to\infty}
        \nu_i\bigl(M_i\setminus B_{g^i_{-\theta}}(x_i,R)\bigr)
        =
        0.
    \end{align*}
    On every fixed scale, the convergence maps identify Euclidean balls with \(g^i_{-\theta}\)-balls around \(x_i\), up to \(o(1)\)-errors. Hence \eqref{eq:entropy-regular-tightness} follows.
\end{proof}

The first consequence is a two-sided comparison with the Euclidean Gaussian profile, uniformly for volumes bounded away from \(0\) and \(1\).

\begin{theorem}[Gaussian-profile regularity]
\label{thm:profile-regularity}
    Fix \(n\ge2\), \(0<\theta<1\), \(0<a_0<1/2\), and \(\eta>0\). There exists \(\delta=\delta(n,\theta,a_0,\eta)>0\) with the following property. If
    \begin{align*}
        \CMcal N_{x,t}(r^2)\ge-\delta,
    \end{align*}
    then, for every \(a\in[a_0,1-a_0]\),
    \begin{align}
        \frac{1}{\sqrt{2\theta}\,r}I(a)
        \le
        \mathcal I_{\nu_\theta}(a)
        \le
        (1+\eta)\frac{1}{\sqrt{2\theta}\,r}I(a).
        \label{eq:two-sided-profile-regularity}
    \end{align}
\end{theorem}

\begin{proof}
    The lower bound is \eqref{eq:universal-profile-lower}. We prove the upper bound by contradiction. By parabolic scaling and time translation, assume \(r=1\) and \(t=0\).
    
    Suppose the upper bound fails. Then there are compact Ricci flows \((M_i,g^i_\tau)_{\tau\in[-1,0]}\), points \(x_i\in M_i\), and numbers \(a_i\in[a_0,1-a_0]\) such that
    \begin{align*}
        \CMcal N^{(i)}_{x_i,0}(1)\to0,
        \qquad
        \nu_i=\nu^{(i)}_{x_i,0;-\theta},
    \end{align*}
    but
    \begin{align}
        \mathcal I_{\nu_i}(a_i)
        >
        (1+\eta)\frac1{\sqrt{2\theta}}I(a_i).
        \label{eq:profile-counterexample}
    \end{align}
    After passing to a subsequence, assume \(a_i\to a_\infty\in[a_0,1-a_0].\) By Proposition~\ref{prop:entropy-regular-Gaussian-convergence}, the pointed flows and the measures \(\nu_i\) converge locally smoothly to the static Euclidean flow and the Gaussian measure \(\gamma_\theta\). Let \(H_c=\{y\in\mathbb R^n:y_1<c\}.\) Choose \(c_\infty\) so that \(\gamma_\theta(H_{c_\infty})=a_\infty.\) Then
    \begin{align*}
        \operatorname{Per}_{\gamma_\theta}(H_{c_\infty})
        =
        \frac1{\sqrt{2\theta}}I(a_\infty).
    \end{align*}
    Choose \(R\) large and \(c_R\) close to \(c_\infty\) so that the truncated half-space \(F_R=H_{c_R}\cap B(0,R)\) satisfies \(\gamma_\theta(F_R)=a_\infty\) and
    \begin{align}
        \operatorname{Per}_{\gamma_\theta}(F_R)
        \le
        \left(1+\frac\eta4\right)
        \frac1{\sqrt{2\theta}}I(a_\infty).
        \label{eq:compact-halfspace-competitor}
    \end{align}
    This is possible because the Gaussian mass of \(H_{c_\infty}\setminus B(0,R)\), the Gaussian perimeter outside large balls, and the spherical cutoff contribution from \(\partial B(0,R)\), all tend to zero as \(R\to\infty\).
    
    Let \(\Psi_i:B(0,R)\to M_i\) be the convergence maps at time \(-\theta\). For \(c\) close to \(c_R\), set \(E_{i,c}=\Psi_i(H_c\cap B(0,R)).\) The local smooth convergence of the metrics and densities gives \(\nu_i(E_{i,c}) \to \gamma_\theta(H_c\cap B(0,R))\) locally uniformly in \(c\). The corresponding weighted perimeters also converge locally uniformly in \(c\). Since \(c\longmapsto \gamma_\theta(H_c\cap B(0,R))\) has positive derivative near \(c_R\), for all large \(i\) we may choose \(c_i\to c_R\) such that \(\nu_i(E_{i,c_i})=a_i.\) Using \eqref{eq:compact-halfspace-competitor} and the continuity of \(I\) on \([a_0,1-a_0]\), we obtain, for all large \(i\),
    \begin{align*}
        \operatorname{Per}_{\nu_i}(E_{i,c_i})
        \le
        \left(1+\frac\eta2\right)
        \frac1{\sqrt{2\theta}}I(a_i).
    \end{align*}
    Thus
    \begin{align*}
        \mathcal I_{\nu_i}(a_i)
        \le
        \operatorname{Per}_{\nu_i}(E_{i,c_i})
        \le
        \left(1+\frac\eta2\right)
        \frac1{\sqrt{2\theta}}I(a_i),
    \end{align*}
    contradicting \eqref{eq:profile-counterexample}. This proves the upper bound.
\end{proof}

We also record the corresponding compactness-stability statement. In the entropy-regular regime, sets with small isoperimetric deficit are close, inside a large regularity chart carrying almost all of the heat-kernel mass, to Euclidean Gaussian half-spaces.

For a finite-perimeter set \(E\subset(M,g_{t-\theta r^2})\), define the scaled deficit
\begin{align}
    \delta_{\nu_\theta}(E)
    =
    \sqrt{2\theta}\,r\,\operatorname{Per}_{\nu_\theta}(E)
    -
    I(\nu_\theta(E)).
\end{align}
The chart in the next theorem is written in the parabolically rescaled coordinates. Equivalently, in the original metric its image has radius of order \(Rr\).

\begin{theorem}[Entropy-regular half-space stability]
\label{thm:entropy-regular-halfspace-stability}
    Fix \(n\ge2\), \(0<\theta<1\), \(0<a_0<1/2\), and \(\eta>0\). There exist
    \begin{align*}
        R=R(n,\theta,a_0,\eta)<\infty,
        \qquad
        \delta_N=\delta_N(n,\theta,a_0,\eta)>0,
        \qquad
        \delta_I=\delta_I(n,\theta,a_0,\eta)>0
    \end{align*}
    with the following property. Let \((M^n,g_t)\) be a compact Ricci flow, let \((x,t)\in M\times I\), and let \(r>0\) satisfy \([t-r^2,t]\subset I\). Assume
    \begin{align*}
        \CMcal N_{x,t}(r^2)\ge-\delta_N.
    \end{align*}
    If \(E\subset(M,g_{t-\theta r^2})\) is a finite-perimeter set with \(\nu_\theta(E)\in[a_0,1-a_0],\) \(\delta_{\nu_\theta}(E)\le\delta_I,\) then there are a smooth entropy-regularity embedding \(\Psi:B_{\mathbb R^n}(0,R)\longrightarrow M\) at time \(t-\theta r^2\), written in the rescaled metric \(r^{-2}g\), and a Euclidean half-space \(H\subset\mathbb R^n\), such that
    \begin{align}
        \begin{aligned}
            \nu_\theta(M\setminus\Psi(B(0,R)))&\le\eta,\\
            \nu_\theta\left(
            (E\cap\Psi(B(0,R)))
            \triangle
            \Psi(H\cap B(0,R))
            \right)
            &\le\eta.
        \end{aligned}
    \end{align}
\end{theorem}

\begin{proof}
    We argue by contradiction. After parabolic scaling and time translation, take \(r=1\) and \(t=0\). Suppose the theorem is false. Then there are compact Ricci flows \((M_i,g^i_\tau)_{\tau\in[-1,0]}\), points \(x_i\in M_i\), numbers \(R_i\to\infty\), and finite-perimeter sets \(E_i\subset(M_i,g^i_{-\theta})\) such that
    \begin{align*}
        \CMcal N^{(i)}_{x_i,0}(1)\to0,
        \qquad
        \nu_i=\nu^{(i)}_{x_i,0;-\theta},\qquad 
        \nu_i(E_i)\in[a_0,1-a_0],
        \qquad
        \sqrt{2\theta}\operatorname{Per}_{\nu_i}(E_i)-I(\nu_i(E_i))\to0,
    \end{align*}
    but the conclusion fails on every entropy-regular convergence chart of radius \(R_i\) and for every Euclidean half-space.
    
    By Proposition~\ref{prop:entropy-regular-Gaussian-convergence}, after passing to a subsequence, the flows converge smoothly on compact subsets to the static Euclidean flow, and the measures \(\nu_i\) converge locally smoothly and tightly to \(d\gamma_\theta(y)=(4\pi\theta)^{-n/2}e^{-|y|^2/4\theta}\,dy.\) Using a diagonal choice of convergence maps, we may assume that \(\Psi_i:B_{\mathbb R^n}(0,R_i)\longrightarrow M_i\) are entropy-regular convergence charts at time \(-\theta\). The deficit bound gives a uniform weighted perimeter bound:
    \begin{align*}
        \operatorname{Per}_{\nu_i}(E_i)\le C(n,\theta,a_0).
    \end{align*}
    On each fixed compact subset of the convergence region, the pulled-back metrics and densities are smoothly comparable to the Euclidean metric and to the Gaussian density. Hence the pulled-back sets have uniform local \(BV\)-bounds. By compactness in \(BV_{\mathrm{loc}}\), after passing to a subsequence, the pulled-back sets converge locally in \(L^1_{\gamma_\theta}\) to a finite-perimeter set \(E_\infty\subset\mathbb R^n.\) The tightness of the measures, together with local \(L^1\)-convergence, gives
    \begin{align*}
        \gamma_\theta(E_\infty)
        =
        \lim_{i\to\infty}\nu_i(E_i)
        \in[a_0,1-a_0].
    \end{align*}
    By lower semicontinuity of weighted perimeter on compact subsets, and then exhausting \(\mathbb R^n\), we have
    \begin{align*}
        \operatorname{Per}_{\gamma_\theta}(E_\infty)
        \le
        \liminf_{i\to\infty}
        \operatorname{Per}_{\nu_i}(E_i).
    \end{align*}
    Using the vanishing deficits, we get
    \begin{align*}
        \sqrt{2\theta}\operatorname{Per}_{\gamma_\theta}(E_\infty)
        \le
        I(\gamma_\theta(E_\infty)).
    \end{align*}
    The reverse inequality is the classical Gaussian isoperimetric inequality. Thus equality holds, and by the equality case in Gaussian isoperimetry, \(E_\infty\) agrees, up to a \(\gamma_\theta\)-null set, with a Euclidean half-space. Call this half-space \(H\).
    
    Choose \(R_0<\infty\) so large that, by tightness,
    \begin{align*}
        \limsup_{i\to\infty}
        \nu_i(M_i\setminus\Psi_i(B(0,R_0)))<\frac{\eta}{10}.
    \end{align*}
    Since the pulled-back sets converge locally in \(L^1_{\gamma_\theta}\) to \(H\), we also have
    \begin{align*}
        \nu_i\left(
        (E_i\cap\Psi_i(B(0,R_0)))
        \triangle
        \Psi_i(H\cap B(0,R_0))
        \right)
        <
        \frac{\eta}{10}
    \end{align*}
    for all large \(i\). Since \(R_i\to\infty\), we have \(R_i\ge R_0\) for all large \(i\). Enlarging the chart from \(R_0\) to \(R_i\) changes the symmetric-difference estimate only inside \(\Psi_i(B(0,R_i)\setminus B(0,R_0)),\) whose \(\nu_i\)-mass is bounded by
    \begin{align*}
        \nu_i(M_i\setminus\Psi_i(B(0,R_0)))<\frac{\eta}{10}
    \end{align*}
    for all large \(i\). Therefore, for all large \(i\),
    \begin{align*}
        \nu_i(M_i\setminus\Psi_i(B(0,R_i)))\le\eta
    \end{align*}
    and
    \begin{align*}
        \nu_i\left(
        (E_i\cap\Psi_i(B(0,R_i)))
        \triangle
        \Psi_i(H\cap B(0,R_i))
        \right)
        \le\eta.
    \end{align*}
    This is exactly the conclusion that was assumed to fail, a contradiction. The theorem follows.
\end{proof}

\section{Log-Sobolev and Poincaré inequalities}
\label{sec:poincare}

We now turn the isoperimetric profile into functional inequalities. The coarea formula gives a first set of \(L^1\)-type consequences, while Gaussian rearrangement transfers the sharp one-dimensional Gaussian constants to the conjugate heat-kernel measure. In this way, the sharp profile gives a unified route to the Hein--Naber log-Sobolev inequality and to the \(L^p\)-Poincaré inequalities with the Gaussian model constants appearing in Bamler's theorem \cite[Theorem~11.1]{bamler2020entropy}.

Throughout this section, \(d\nu(y)=K(x_0,t_0;y,s)\,dg_s(y),\) \(\tau=t_0-s.\) Let \(\gamma_\tau\) denote the one-dimensional Gaussian measure \(N(0,2\tau)\). Its isoperimetric profile is \((2\tau)^{-1/2}I\).

\begin{proposition}[Coarea profile]
\label{prop:coarea-profile}
    For every locally Lipschitz function \(h:M\to\mathbb R\),
    \begin{align}
        \int_M |\nabla h|\,d\nu
        \ge
        \frac1{\sqrt{2\tau}}
        \int_{-\infty}^{\infty}
        I(\nu(\{h>r\}))\,dr.
        \label{eq:coarea-profile}
    \end{align}
\end{proposition}

\begin{proof}
    The weighted coarea formula gives
    \begin{align*}
        \int_M|\nabla h|\,d\nu
        =
        \int_{-\infty}^{\infty}
        \operatorname{Per}_\nu(\{h>r\})\,dr.
    \end{align*}
    Applying Theorem~\ref{thm:intro-GI} to the superlevel sets gives the claim.
\end{proof}

The first consequence is a support-sensitive Poincaré inequality. It improves as the support becomes small, reflecting the growth of the Gaussian isoperimetric ratio \(I(a)/a\) near \(a=0\).

\begin{corollary}[Small-support Poincaré inequality]
\label{cor:small-support-poincare}
    Let \(u\ge0\) be locally Lipschitz and assume \(\nu(\{u>0\})\le a\le\frac12.\) Then, for every \(p\ge1\),
    \begin{align}
        \|u\|_{L^p(\nu)}
        \le
        p\sqrt{2\tau}\frac{a}{I(a)}
        \|\nabla u\|_{L^p(\nu)}.
        \label{eq:small-support-poincare}
    \end{align}
    In particular, if \(m\) is a median of a locally Lipschitz function \(h\), then
    \begin{align}
        \|h-m\|_{L^p(\nu)}
        \le
        p\sqrt{\pi\tau}\,\|\nabla h\|_{L^p(\nu)}.
        \label{eq:median-poincare}
    \end{align}
\end{corollary}

\begin{proof}
    If \(a=0\), then \(u=0\) \(\nu\)-a.e., and there is nothing to prove. Assume \(0<a\le1/2\). Since \(I\) is concave and \(I(0)=0\), the function \(b\mapsto I(b)/b\) is nonincreasing. Hence \(b\le \frac{a}{I(a)}I(b),\) \(0<b\le a.\) Applying this to the superlevel measures of \(u\), the layer-cake formula and Proposition~\ref{prop:coarea-profile} give
    \begin{align*}
        \int_M u\,d\nu
        =
        \int_0^\infty \nu(\{u>r\})\,dr
        \le
        \sqrt{2\tau}\frac{a}{I(a)}
        \int_M|\nabla u|\,d\nu.
    \end{align*}
    Apply this estimate to \(u^p\). Since \(|\nabla u^p|=pu^{p-1}|\nabla u|,\) Hölder inequality gives
    \begin{align*}
        \int_M u^p\,d\nu
        \le
        p\sqrt{2\tau}\frac{a}{I(a)}
        \|u\|_{L^p(\nu)}^{p-1}\|\nabla u\|_{L^p(\nu)}.
    \end{align*}
    This proves \eqref{eq:small-support-poincare}.
    
    For the median estimate, apply the result with \(a=1/2\) to \((h-m)_+\) and \((h-m)_-\), and use \(I(1/2)=\frac1{\sqrt{2\pi}}.\) Since
    \begin{align*}
        |\nabla(h-m)_+|^p+|\nabla(h-m)_-|^p=|\nabla h|^p
    \end{align*}
    a.e., this gives \eqref{eq:median-poincare}.
\end{proof}

We next use the profile to compare functions with their one-dimensional Gaussian rearrangements. This is the step that transfers sharp Gaussian functional inequalities to the conjugate heat-kernel measure. If \(h:M\to\mathbb R\) is measurable, let \(h^\circ:\mathbb R\to\mathbb R\) be the nonincreasing rearrangement of \(h\) with respect to \(\gamma_\tau\), chosen so that \(h\) and \(h^\circ\) are equimeasurable:
\begin{align*}
    \gamma_\tau(\{h^\circ>r\})=\nu(\{h>r\})
\end{align*}
for every continuity value \(r\) of the distribution function. Equivalently, all distributional quantities of \(h\) and \(h^\circ\) agree.

\begin{theorem}[Gaussian Pólya--Szegő inequality]
\label{thm:gaussian-polya-szego}
    For every locally Lipschitz \(h:M\to\mathbb R\) and every \(p\ge1\),
    \begin{align}
        \int_{\mathbb R}|(h^\circ)'|^p\,d\gamma_\tau
        \le
        \int_M|\nabla h|^p\,d\nu.
        \label{eq:gaussian-polya-szego}
    \end{align}
\end{theorem}

\begin{proof}
    We first prove the result for smooth \(h\) with regular level sets. The general locally Lipschitz case follows by the usual approximation and truncation argument.
    
    Put \(\mu(r)=\nu(\{h>r\}).\) For a.e. \(r\), the weighted coarea formula gives
    \begin{align*}
        -\mu'(r)
        =
        \int_{\{h=r\}}\frac1{|\nabla h|}\,dA_\nu,
        \qquad
        \operatorname{Per}_\nu(\{h>r\})
        =
        \int_{\{h=r\}}1\,dA_\nu,
    \end{align*}
    where \(dA_\nu = K(x_0,t_0;\cdot,s)\,d\mathcal H^{n-1}_{g_s}.\) For \(p>1\), Hölder inequality on the level set \(\{h=r\}\) gives
    \begin{align*}
        \int_{\{h=r\}}|\nabla h|^{p-1}\,dA_\nu
        \ge
        \frac{\operatorname{Per}_\nu(\{h>r\})^p}{(-\mu'(r))^{p-1}}.
    \end{align*}
    Integrating in \(r\) and using Theorem~\ref{thm:intro-GI}, we obtain
    \begin{align}
        \int_M|\nabla h|^p\,d\nu
        \ge
        \int_{-\infty}^{\infty}
        \frac{\left((2\tau)^{-1/2}I(\mu(r))\right)^p}
        {(-\mu'(r))^{p-1}}\,dr.
        \label{eq:polya-intermediate}
    \end{align}
    
    For \(h^\circ\), the superlevel sets are Gaussian half-lines of measure \(\mu(r)\), and their Gaussian perimeter is exactly \((2\tau)^{-1/2}I(\mu(r)).\) The one-dimensional coarea identity therefore identifies the right-hand side of \eqref{eq:polya-intermediate} with \(\int_{\mathbb R}|(h^\circ)'|^p\,d\gamma_\tau.\) This proves the result for \(p>1\). The case \(p=1\) is the same argument without the Hölder step.
\end{proof}

As a first application of rearrangement, we obtain the sharp log-Sobolev inequality of Hein--Naber. The proof below is included to emphasize that the constant comes directly from the one-dimensional Gaussian model through the Pólya--Szegő inequality.

For a nonnegative integrable function \(w\), write
\begin{align*}
    \operatorname{Ent}_{\nu}(w)
    =
    \int_M w\log w\,d\nu
    -
    \left(\int_M w\,d\nu\right)
    \log\left(\int_M w\,d\nu\right).
\end{align*}

\begin{proposition}[{\cite[Theorem 1.10]{MR3245102}}]
\label{prop:GI-implies-HN-LSI}
    Let \(d\nu=K(x_0,t_0;\cdot,s)\,dg_s,\) \(\tau=t_0-s.\) Then every smooth function \(u:M\to\mathbb R\) satisfies
    \begin{align}
        \operatorname{Ent}_{\nu}(u^2)
        \le
        4\tau\int_M|\nabla u|_{g_s}^2\,d\nu.
        \label{eq:HN-LSI-u2}
    \end{align}
    Equivalently, every positive smooth function \(w\) satisfies
    \begin{align}
        \operatorname{Ent}_{\nu}(w)
        \le
        \tau\int_M\frac{|\nabla w|_{g_s}^2}{w}\,d\nu.
        \label{eq:HN-LSI-w}
    \end{align}
\end{proposition}

\begin{proof}
    Since \(u^2=|u|^2\) and \(|\nabla |u||\le |\nabla u|\), it suffices to prove \eqref{eq:HN-LSI-u2} for \(u\ge0\).
    
    Let \(u^\circ:\mathbb R\to[0,\infty)\) be the monotone Gaussian rearrangement of \(u\) with respect to \(\gamma_\tau\). Since \(u\) and \(u^\circ\) are equimeasurable,
    \begin{align}
        \operatorname{Ent}_{\nu}(u^2)
        =
        \operatorname{Ent}_{\gamma_\tau}((u^\circ)^2).
        \label{eq:entropy-equimeasurable-short}
    \end{align}
    By Theorem~\ref{thm:gaussian-polya-szego}, applied with \(p=2\),
    \begin{align}
        \int_{\mathbb R}|(u^\circ)'|^2\,d\gamma_\tau
        \le
        \int_M|\nabla u|^2\,d\nu.
        \label{eq:polya-p2-for-LSI}
    \end{align}
    The sharp one-dimensional Gaussian log-Sobolev inequality for \(\gamma_\tau=N(0,2\tau)\) gives
    \begin{align*}
        \operatorname{Ent}_{\gamma_\tau}((u^\circ)^2)
        \le
        4\tau
        \int_{\mathbb R}|(u^\circ)'|^2\,d\gamma_\tau.
    \end{align*}
    Together with \eqref{eq:entropy-equimeasurable-short} and \eqref{eq:polya-p2-for-LSI}, this proves \eqref{eq:HN-LSI-u2}.
    
    Finally, if \(w>0\), apply \eqref{eq:HN-LSI-u2} to \(u=\sqrt w\). Since \(|\nabla\sqrt w|^2=\frac{|\nabla w|^2}{4w},\) we obtain
    \begin{align*}
        \operatorname{Ent}_{\nu}(w)
        \le
        4\tau\int_M|\nabla\sqrt w|^2\,d\nu
        =
        \tau\int_M\frac{|\nabla w|^2}{w}\,d\nu.
    \end{align*}
\end{proof}

We also record a local reverse form of the log-Sobolev inequality. Unlike the usual Hein--Naber log-Sobolev inequality, this estimate controls the endpoint gradient of the heat evolution by the entropy at the earlier time.

\begin{proposition}[Local reverse log-Sobolev inequality]
\label{prop:reverse-LSI}
    Fix \(x\in M\), \(s<t\), put \(\tau=t-s\), and let \(d\nu(y)=K(x,t;y,s)\,dg_s(y).\) Then every positive smooth function \(f\) satisfies
    \begin{align}
        \operatorname{Ent}_{\nu}(f)
        \ge
        \tau\,
        \frac{
        |\nabla_x P_{s,t}f|_{g_t}^2(x)
        }{
        P_{s,t}f(x)
        }.
        \label{eq:reverse-LSI}
    \end{align}
\end{proposition}

\begin{proof}
    Let \(u(\cdot,r)=P_{s,r}f\), so that \((\partial_r-\Delta_{g_r})u=0.\) For \(r\in[s,t]\), write \(d\nu_r(y)=K(x,t;y,r)\,dg_r(y).\) Using the adjoint identity with \(h=u\log u\), we get
    \begin{align*}
        \frac{d}{dr}\int_M u\log u\,d\nu_r
        =
        \int_M(\partial_r-\Delta_{g_r})(u\log u)\,d\nu_r
        =
        -\int_M\frac{|\nabla u|_{g_r}^2}{u}\,d\nu_r.
    \end{align*}
    Since \(\nu_s=\nu\) and \(\nu_t=\delta_x\), integration from \(s\) to \(t\) gives
    \begin{align}
        \operatorname{Ent}_{\nu}(f)
        =
        \int_s^t
        P_{r,t}\left(\frac{|\nabla u|_{g_r}^2}{u}\right)(x)\,dr.
        \label{eq:entropy-identity-reverse-LSI}
    \end{align}
    Set \(a=P_{s,t}f(x).\) For every \(r\in[s,t]\), the semigroup property gives \(P_{r,t}u(x)=a.\) By Cauchy inequality with respect to \(\nu_{x,t;r}\),
    \begin{align*}
        P_{r,t}(|\nabla u|)(x)^2
        \le
        P_{r,t}u(x)\,
        P_{r,t}\left(\frac{|\nabla u|^2}{u}\right)(x)
        =
        a\,
        P_{r,t}\left(\frac{|\nabla u|^2}{u}\right)(x).
    \end{align*}
    On the other hand, the gradient contraction \eqref{eq:gradient-contraction}, applied from time \(r\) to time \(t\), gives
    \begin{align*}
        |\nabla_xP_{r,t}u|_{g_t}(x)
        \le
        P_{r,t}(|\nabla u|)(x).
    \end{align*}
    Since \(P_{r,t}u=P_{s,t}f\), we obtain
    \begin{align*}
        P_{r,t}\left(\frac{|\nabla u|^2}{u}\right)(x)
        \ge
        \frac{
        |\nabla_xP_{s,t}f|_{g_t}^2(x)
        }{
        P_{s,t}f(x)
        }.
    \end{align*}
    Substituting this into \eqref{eq:entropy-identity-reverse-LSI} gives \eqref{eq:reverse-LSI}.
\end{proof}

Let \(\gamma\) be the standard one-dimensional Gaussian measure. For \(p\ge1\), define
\begin{align}
    \Lambda_p
    =
    \sup_{g\not\equiv{\rm const}}
    \frac{
    \left\|g-\int g\,d\gamma\right\|_{L^p(\gamma)}
    }{
    \|g'\|_{L^p(\gamma)}
    }.
\end{align}
Thus \(\Lambda_p\) is the sharp Gaussian \(L^p\)-Poincaré constant on the line. We write
\begin{align*}
    \nu h=\int_M h\,d\nu.
\end{align*}

The same rearrangement argument gives the \(L^p\)-Poincaré inequalities with the sharp one-dimensional Gaussian model constants \cite[Theorem~11.1]{bamler2020entropy}.

\begin{theorem}[Gaussian-model \(L^p\)-Poincaré inequality]
\label{thm:sharp-poincare}
    For every locally Lipschitz \(h:M\to\mathbb R\) and every \(p\ge1\),
    \begin{align}
        \|h-\nu h\|_{L^p(\nu)}
        \le
        \sqrt{2\tau}\,\Lambda_p\,\|\nabla h\|_{L^p(\nu)}.
        \label{eq:sharp-poincare}
    \end{align}
    The constant is sharp as a universal Ricci-flow inequality.
\end{theorem}

\begin{proof}
    By equimeasurability,
    \begin{align*}
        \|h-\nu h\|_{L^p(\nu)}
        =
        \|h^\circ-\gamma_\tau h^\circ\|_{L^p(\gamma_\tau)}.
    \end{align*}
    Scaling the sharp one-dimensional inequality from \(N(0,1)\) to \(N(0,2\tau)\) gives
    \begin{align*}
        \|g-\gamma_\tau g\|_{L^p(\gamma_\tau)}
        \le
        \sqrt{2\tau}\,\Lambda_p\|g'\|_{L^p(\gamma_\tau)}.
    \end{align*}
    Apply this to \(g=h^\circ\), and then use Theorem~\ref{thm:gaussian-polya-szego}.
    
    Sharpness follows from the static Euclidean flow. There \(\nu\) is a Gaussian measure with covariance \(2\tau\operatorname{Id}_n\), and functions depending on one Euclidean coordinate reduce the inequality to the one-dimensional model defining \(\Lambda_p\).
\end{proof}

\begin{remark}
    For the Gaussian line, \(\Lambda_1=\sqrt{\frac{\pi}{2}},\) \(\Lambda_2=1.\) Hence Theorem~\ref{thm:sharp-poincare} gives
    \begin{align*}
        \|h-\nu h\|_{L^1(\nu)} \le \sqrt{\pi\tau}\,\|\nabla h\|_{L^1(\nu)}, \qquad \int_M |h-\nu h|^2\,d\nu \le 2\tau\int_M|\nabla h|^2\,d\nu.
    \end{align*}
    These are the sharp \(p=1\) and \(p=2\) constants in Bamler's normalization. In general, the optimal universal Ricci-flow constant is exactly the one-dimensional Gaussian model constant.
\end{remark}

\subsection{Cheeger and Faber--Krahn estimates}

We finish the section with two elementary consequences of the same profile: a weighted Cheeger lower bound and a support-sensitive Faber--Krahn estimate.

\begin{corollary}[Weighted Cheeger constant]
\label{cor:weighted-cheeger}
    Let
    \begin{align*}
        h_\nu
        =
        \inf_{0<\nu(E)<1}
        \frac{\operatorname{Per}_\nu(E)}
        {\min\{\nu(E),1-\nu(E)\}}.
    \end{align*}
    Then \(h_\nu\ge \frac1{\sqrt{\pi\tau}}.\)
\end{corollary}

\begin{proof}
    By symmetry, consider \(a=\nu(E)\le1/2\). Theorem~\ref{thm:intro-GI} gives
    \begin{align*}
        \frac{\operatorname{Per}_\nu(E)}{\nu(E)}
        \ge
        \frac1{\sqrt{2\tau}}\frac{I(a)}{a}.
    \end{align*}
    Since \(I\) is concave and \(I(0)=0\), the ratio \(I(a)/a\) is decreasing. Thus
    \begin{align*}
        \frac{I(a)}{a}
        \ge
        \frac{I(1/2)}{1/2}
        =
        \sqrt{\frac2\pi}.
    \end{align*}
    This proves the estimate.
\end{proof}

\begin{corollary}[Weighted Faber--Krahn estimate]
\label{cor:weighted-faber-krahn}
    Let \(\Omega\subset M\) be open with \(\nu(\Omega)=a\le1/2\), and define
    \begin{align*}
        \lambda_1^\nu(\Omega)
        =
        \inf_{\phi\in C_c^\infty(\Omega),\,\phi\not\equiv0}
        \frac{\int_\Omega |\nabla\phi|^2\,d\nu}
        {\int_\Omega \phi^2\,d\nu}.
    \end{align*}
    Then
    \begin{align}
        \lambda_1^\nu(\Omega)
        \ge
        \frac1{8\tau}\left(\frac{I(a)}{a}\right)^2.
        \label{eq:weighted-faber-krahn}
    \end{align}
    In particular,
    \begin{align*}
        \lambda_1^\nu(\Omega)
        \ge
        \frac{\log(1/a)+o(\log(1/a))}{4\tau}
        \qquad\text{as }a\downarrow0.
    \end{align*}
\end{corollary}

\begin{proof}
    For every finite-perimeter set \(E\subset\Omega\), we have \(\nu(E)\le a\). Since \(a\le1/2\) and \(I(b)/b\) is decreasing,
    \begin{align*}
        \frac{\operatorname{Per}_\nu(E)}{\nu(E)}
        \ge
        \frac1{\sqrt{2\tau}}\frac{I(a)}{a}.
    \end{align*}
    Thus the Dirichlet--Cheeger constant of \(\Omega\) is at least \(\frac1{\sqrt{2\tau}}\frac{I(a)}{a}.\) Cheeger's inequality gives
    \begin{align*}
        \lambda_1^\nu(\Omega)
        \ge
        \frac14
        \left(
        \frac1{\sqrt{2\tau}}\frac{I(a)}{a}
        \right)^2
        =
        \frac1{8\tau}
        \left(\frac{I(a)}{a}\right)^2.
    \end{align*}
    The final asymptotic follows from \(I(a)\sim a\sqrt{2\log(1/a)}\) as \(a\downarrow0.\)
\end{proof}

\section{Reverse hypercontractivity}
\label{sec:reverse-hypercontractivity}

Fix a spacetime point \((x_0,t_0)\). For \(t<t_0\), write \(d\nu_t(y)=K(x_0,t_0;y,t)\,dg_t(y).\) Bamler's hypercontractivity theorem treats the range \(1<q\le p<\infty\) \cite[Theorem 12.1]{bamler2020entropy}. We record here the corresponding reverse inequality in the range \(0<p\le q<1\). The proof is the usual argument due to Gross \cite{gross-1975-log-sobolev}, with the sign reversed because the exponent is below one.

We shall use the following standard differential identity. If \(u>0\) solves \((\partial_t-\Delta_{g_t})u=0\) and \(\alpha(t)>0\), then
\begin{align}\label{eq:gross-identity}
\begin{aligned}
    \frac{d}{dt}\log \|u(\cdot,t)\|_{L^{\alpha(t)}(\nu_t)}
    =
    \frac{1}{\alpha(t)^2\int u^{\alpha(t)}\,d\nu_t}
    \bigg[
    &\alpha'(t)\operatorname{Ent}_{\nu_t}(u^{\alpha(t)})      \\
    &-(\alpha(t)-1)
    \int_M\frac{|\nabla u^{\alpha(t)}|^2}{u^{\alpha(t)}}\,d\nu_t
    \bigg].
\end{aligned}     
\end{align}
It follows by differentiating \(\alpha(t)^{-1}\log\int_M u^{\alpha(t)}\,d\nu_t\), using \(\frac{d}{dt}\int_M h\,d\nu_t = \int_M \Box h\,d\nu_t.\)

\begin{theorem}[Reverse hypercontractivity]
\label{thm:reverse-hypercontractivity}
    Let \(0<\tau_1<\tau_2,\) \(s_i=t_0-\tau_i.\) Let \(u>0\) be a smooth positive solution of the heat equation on \(M\times[s_2,s_1]\). If \(0<p\le q<1\) and
    \begin{align}
        \frac{\tau_2}{\tau_1}
        \ge
        \frac{1-p}{1-q},
        \label{eq:reverse-threshold}
    \end{align}
    then
    \begin{align}
        \|u(\cdot,s_1)\|_{L^p(\nu_{s_1})}
        \ge
        \|u(\cdot,s_2)\|_{L^q(\nu_{s_2})}.
        \label{eq:reverse-hypercontractivity}
    \end{align}
\end{theorem}

\begin{proof}
    We first prove the equality case in \eqref{eq:reverse-threshold}. Put \(c=(1-q)\tau_2,\) \(\alpha(t)=1-\frac{c}{t_0-t}.\) Then \(\alpha(s_2)=q\), \(\alpha(s_1)=p\), and
    \begin{align}
        \alpha'(t)<0,
        \qquad
        \alpha'(t)(t_0-t)=\alpha(t)-1.
        \label{eq:alpha-relation}
    \end{align}
    Using \eqref{eq:gross-identity} and then the log-Sobolev inequality \eqref{eq:HN-LSI-w}, we get
    \begin{align*}
        \alpha'(t)\operatorname{Ent}_{\nu_t}(u^{\alpha(t)})
        \ge
        \alpha'(t)(t_0-t)
        \int_M
        \frac{|\nabla u^{\alpha(t)}|^2}{u^{\alpha(t)}}\,d\nu_t.
    \end{align*}
    By \eqref{eq:alpha-relation}, the bracket in \eqref{eq:gross-identity} is nonnegative. Hence \(t\longmapsto \|u(\cdot,t)\|_{L^{\alpha(t)}(\nu_t)}\) is nondecreasing on \([s_2,s_1]\). Evaluating at the endpoints gives \eqref{eq:reverse-hypercontractivity} in the equality case.
    
    Now assume only \eqref{eq:reverse-threshold}. Set \(\tau_1'=\tau_2\frac{1-q}{1-p}.\) Then \(\tau_1\le\tau_1'\). If \(\tau_1'<\tau_2\), the equality case applied on \([t_0-\tau_2,t_0-\tau_1']\) gives
    \begin{align*}
        \|u(\cdot,t_0-\tau_1')\|_{L^p(\nu_{t_0-\tau_1'})}
        \ge
        \|u(\cdot,s_2)\|_{L^q(\nu_{s_2})}.
    \end{align*}
    If \(\tau_1'=\tau_2\), the same inequality is simply the identity with \(p=q\).
    
    It remains to move from \(t_0-\tau_1'\) to \(s_1\) with the fixed exponent \(p<1\). Since
    \begin{align*}
        \frac{d}{dt}\int_M u^p\,d\nu_t
        =
        -p(p-1)\int_M u^{p-2}|\nabla u|^2\,d\nu_t
        \ge0,
    \end{align*}
    the quantity \(\|u(\cdot,t)\|_{L^p(\nu_t)}\) is nondecreasing in \(t\). Therefore
    \begin{align*}
        \|u(\cdot,s_1)\|_{L^p(\nu_{s_1})}
        \ge
        \|u(\cdot,t_0-\tau_1')\|_{L^p(\nu_{t_0-\tau_1'})},
    \end{align*}
    and the theorem follows.
\end{proof}

\begin{remark}
    The threshold in \eqref{eq:reverse-threshold} is forced by the flat Gaussian model. In the static Euclidean flow on \(\mathbb R\), based at \((0,0)\), one has \(d\nu_{-\tau}(y)=(4\pi\tau)^{-1/2}e^{-y^2/4\tau}\,dy\). Taking \(u(y,-\tau_2)=e^{\lambda y}\), the heat evolution gives \(u(y,-\tau_1)=e^{\lambda^2(\tau_2-\tau_1)}e^{\lambda y}\), while \(\|e^{\lambda y}\|_{L^a(\nu_{-\tau})} = e^{a\lambda^2\tau}.\) Thus
    \begin{align*}
        \|u(\cdot,-\tau_1)\|_{L^p(\nu_{-\tau_1})}
        =
        \exp\left(\lambda^2(\tau_2+(p-1)\tau_1)\right),
        \qquad
        \|u(\cdot,-\tau_2)\|_{L^q(\nu_{-\tau_2})}
        =
        \exp(q\lambda^2\tau_2).
    \end{align*}
    If \eqref{eq:reverse-hypercontractivity} holds for all \(\lambda\), then \(\tau_2+(p-1)\tau_1\ge q\tau_2\), equivalently \(\frac{\tau_2}{\tau_1}\ge\frac{1-p}{1-q}.\)
\end{remark}

\section{A foray into the path-space}
\label{sec:path-space-bobkov}

We finish with a path-space version of Bobkov inequality. The path-space inequality is obtained by tensorizing the finite-dimensional inequality along the Markov structure of parabolic Wiener measure, in the spirit of Bakry--Ledoux \cite{MR1374200}. The only geometric input beyond the finite-dimensional Bobkov inequality is the Haslhofer--Naber path-space gradient estimate \((R2)\) \cite{MR3790068}, which controls the spatial gradient of a conditional expectation by the future parallel gradient. The argument is as follows. For a cylinder function, we split the path at its first observation time. The future part is estimated by induction, while the first time-slice is estimated by the finite-dimensional Bobkov inequality for the corresponding conjugate heat-kernel measure. The Haslhofer--Naber gradient estimate then allows these two contributions to be recombined into the full \(\mathcal H\)-gradient. This gives a Bobkov inequality on parabolic path space and, by relaxation, the corresponding Gaussian isoperimetric inequality for path-space sets.

Let \((M^n,g_t)_{t\in[0,T]}\) be a smooth closed Ricci flow. Fix a spacetime point \((x,T)\). Let \(P_{(x,T)}M\) denote the parabolic path space based at \((x,T)\). A path is written
\begin{align*}
    \omega=\{\omega_\sigma\}_{0\le\sigma\le T},
    \qquad
    \omega_\sigma=(X_\sigma,T-\sigma),
    \qquad
    X_0=x.
\end{align*}
Let \(\Gamma_{(x,T)}\) be the corresponding Wiener measure. In the normalization of Haslhofer--Naber \cite{MR3790068}, \(X_\sigma\) is Brownian motion run backward through the evolving metrics \(g_{T-\sigma}\), with generator \(\Delta_{g_{T-\sigma}}\).

For a smooth cylinder function \(F(\omega)=u(X_{\tau_1},\ldots,X_{\tau_k}),\) \(0<\tau_1<\cdots<\tau_k\le T,\) let \(\mathcal P_\sigma:T_{X_\sigma}M\longrightarrow T_xM\) denote stochastic parallel transport. We use the Haslhofer--Naber convention for parallel gradients:
\begin{align*}
    \nabla^\parallel_\sigma F
    =
    \sum_{\tau_i\ge\sigma}
    \mathcal P_{\tau_i}\nabla_i u,
    \qquad
    0\le\sigma\le T.
\end{align*}
The Malliavin \(\mathcal H\)-gradient is normalized by
\begin{align*}
    |\nabla^{\mathcal H}F|_{\mathcal H}^2
    =
    \int_0^T|\nabla^\parallel_\sigma F|^2\,d\sigma.
\end{align*}
For one-time cylinder functions \(F(\omega)=f(X_\tau)\), this gives
\begin{align}\label{eq:one-time-H-gradient}
    |\nabla^{\mathcal H}F|_{\mathcal H}^2 = \tau|\nabla f|_{g_{T-\tau}}^2(X_\tau). 
\end{align}
We write \(\Sigma_\sigma\) for the \(\sigma\)-algebra generated by the path up to time \(\sigma\). Unless otherwise indicated, all expectations are taken with respect to \(\Gamma_{(x,T)}\).

We shall use the following form of the Haslhofer--Naber gradient estimate \((R2)\). If \((M,g_t)\) evolves by Ricci flow, then for every cylinder function \(F\) on the path space based at \((y,T')\),
\begin{align}
    \left|
    \nabla_y
    \mathbb E_{(y,T')}[F]
    \right|
    \le
    \mathbb E_{(y,T')}
    \left[
    |\nabla^\parallel_0F|
    \right].
    \label{eq:HN-R2-used} \tag{R2}
\end{align}
Here \(\nabla^\parallel\) denotes the parallel gradient on the based path space under consideration.

\begin{lemma}[First-step gradient estimate]
\label{lem:first-step-gradient}
    Let \(F(\omega)=u(X_{\tau_1},\ldots,X_{\tau_k}),\) \(0<\tau_1<\cdots<\tau_k\le T,\) and set \(G=\mathbb E_{(x,T)}[F\mid\Sigma_{\tau_1}].\) Then \(G=g(X_{\tau_1})\) for a smooth function \(g\), and
    \begin{align}
        |\nabla g|_{g_{T-\tau_1}}(X_{\tau_1})
        \le
        \mathbb E_{(x,T)}
        \left[
        |\nabla^\parallel_{\tau_1}F|
        \,\middle|\,
        \Sigma_{\tau_1}
        \right].
        \label{eq:first-step-gradient}
    \end{align}
\end{lemma}

\begin{proof}
    By the Markov property, \(G=g(X_{\tau_1}),\) where \(g(y) = \mathbb E_{(y,T-\tau_1)} \left[u\left(y,Y_{\tau_2-\tau_1},\ldots,Y_{\tau_k-\tau_1}\right) \right].\) Equivalently, on the path space based at \((y,T-\tau_1)\), define
    \begin{align*}
        \widehat F(\eta)
        =
        u\left(
        \eta_0,\eta_{\tau_2-\tau_1},\ldots,\eta_{\tau_k-\tau_1}
        \right).
    \end{align*}
    Then \(g(y)=\mathbb E_{(y,T-\tau_1)}[\widehat F].\) Applying \eqref{eq:HN-R2-used} on the future path space based at \((y,T-\tau_1)\), we get
    \begin{align*}
        |\nabla g|_{g_{T-\tau_1}}(y)
        \le
        \mathbb E_{(y,T-\tau_1)}
        \left[
        |\nabla^\parallel_0\widehat F|
        \right].
    \end{align*}
    Under concatenation of the initial segment \(\omega|_{[0,\tau_1]}\) with a future path based at \((X_{\tau_1},T-\tau_1)\), the Haslhofer--Naber formula for parallel gradients gives
    \begin{align*}
        \mathcal P_{\tau_1}
        \bigl(\nabla^\parallel_0\widehat F\bigr)
        =
        \nabla^\parallel_{\tau_1}F.
    \end{align*}
    Since \(\mathcal P_{\tau_1}\) is an isometry, the norms agree. Evaluating at \(y=X_{\tau_1}\) and conditioning on \(\Sigma_{\tau_1}\) proves \eqref{eq:first-step-gradient}.
\end{proof}

We now prove the path-space Bobkov inequality. The normalization is chosen so that one-time cylinder functions recover the finite-dimensional inequality.

\begin{theorem}[Path-space Bobkov inequality]
\label{thm:path-space-bobkov}
    For every smooth cylinder function \(F:P_{(x,T)}M\longrightarrow[0,1],\) one has
    \begin{align}
        I\left(
        \int_{P_{(x,T)}M}F\,d\Gamma_{(x,T)}
        \right)
        \le
        \int_{P_{(x,T)}M}
        \sqrt{
        I(F)^2+2|\nabla^{\mathcal H}F|_{\mathcal H}^2
        }
        \,d\Gamma_{(x,T)}.
        \label{eq:path-space-bobkov}
    \end{align}
    The constant \(2\) is sharp as a universal constant.
\end{theorem}

\begin{proof}
    We argue by induction on the number of positive time-slices on which the cylinder function depends. First suppose \(F(\omega)=f(X_\tau).\) The pushforward of \(\Gamma_{(x,T)}\) under \(X_\tau\) is the conjugate heat-kernel measure \(\nu_{x,T;T-\tau}\). Hence Theorem~\ref{thm:Bobkov} gives
    \begin{align*}
        I\left(\int F\,d\Gamma_{(x,T)}\right)
        \le
        \int
        \sqrt{
        I(f)^2+2\tau|\nabla f|_{g_{T-\tau}}^2
        }
        \,d\nu_{x,T;T-\tau}.
    \end{align*}
    By \eqref{eq:one-time-H-gradient}, this is exactly \eqref{eq:path-space-bobkov}. This proves the base case.
    
    Now assume the result for cylinder functions depending on at most \(k-1\) positive time-slices, and let \(F(\omega)=u(X_{\tau_1},\ldots,X_{\tau_k}),\) \(0<\tau_1<\cdots<\tau_k\le T.\) Set \(G=\mathbb E[F\mid\Sigma_{\tau_1}]=g(X_{\tau_1}).\) The pushforward of \(\Gamma_{(x,T)}\) under \(X_{\tau_1}\) is \(\nu_{x,T;T-\tau_1}\). Applying the finite-dimensional Bobkov inequality to \(g\), we get
    \begin{align}
        I(\mathbb E F)
        =
        I(\mathbb E G)
        \le
        \mathbb E
        \sqrt{
        I(G)^2+2\tau_1|\nabla g|_{g_{T-\tau_1}}^2(X_{\tau_1})
        }.
        \label{eq:path-first-step}
    \end{align}
    
    We next estimate the two terms inside the square root. Fix \(y\in M\), and consider the future path space based at \((y,T-\tau_1)\). For the induction step, we freeze the first variable and define
    \begin{align*}
        F_y^+(\eta)
        =
        u\left(
        y,\eta_{\tau_2-\tau_1},\ldots,\eta_{\tau_k-\tau_1}
        \right).
    \end{align*}
    Then \(g(y) = \mathbb E_{(y,T-\tau_1)}[F_y^+].\) Applying the induction hypothesis to \(F_y^+\), and then evaluating at \(y=X_{\tau_1}\), gives
    \begin{align}
        I(G)
        \le
        \mathbb E
        \left[
        \mathcal A
        \,\middle|\,
        \Sigma_{\tau_1}
        \right],
        \label{eq:path-induction}
    \end{align}
    where
    \begin{align}
        \mathcal A
        \coloneqq 
        \sqrt{
        I(F)^2
        +
        2\int_{\tau_1}^T
        |\nabla^\parallel_\sigma F|^2\,d\sigma
        }.
        \label{eq:path-A-definition}
    \end{align}
    Here we used the compatibility of the future \(\mathcal H\)-gradient under concatenation: after the first time \(\tau_1\), the future parallel gradients are precisely the restrictions of \(\nabla^\parallel_\sigma F\) to \(\sigma\in[\tau_1,T]\).
    
    On the other hand, by Lemma~\ref{lem:first-step-gradient},
    \begin{align}
        |\nabla g|_{g_{T-\tau_1}}(X_{\tau_1})
        \le
        \mathbb E
        \left[
        |\nabla^\parallel_{\tau_1}F|
        \,\middle|\,
        \Sigma_{\tau_1}
        \right].
        \label{eq:path-gradient-step}
    \end{align}
    Combining \eqref{eq:path-induction} and \eqref{eq:path-gradient-step}, and using Minkowski inequality in \(\mathbb R^2\), gives
    \begin{align}\label{eq:path-minkowski}
    \begin{aligned}
        \sqrt{
        I(G)^2+2\tau_1|\nabla g|^2(X_{\tau_1})
        } 
        &\le
        \sqrt{
        \left(\mathbb E[\mathcal A\mid\Sigma_{\tau_1}]\right)^2
        +
        \left(
        \mathbb E[
        \sqrt{2\tau_1}|\nabla^\parallel_{\tau_1}F|
        \mid\Sigma_{\tau_1}]
        \right)^2
        } \\
        &\le
        \mathbb E\left[
        \sqrt{
        \mathcal A^2+2\tau_1|\nabla^\parallel_{\tau_1}F|^2
        }
        \,\middle|\,\Sigma_{\tau_1}
        \right].
    \end{aligned}
    \end{align}
    For \(0\le\sigma\le\tau_1\), up to the irrelevant endpoint convention, \(\nabla^\parallel_\sigma F = \nabla^\parallel_{\tau_1}F.\) Therefore
    \begin{align*}
        \mathcal A^2+2\tau_1|\nabla^\parallel_{\tau_1}F|^2
        &=
        I(F)^2
        +
        2\int_{\tau_1}^T|\nabla^\parallel_\sigma F|^2\,d\sigma
        +
        2\int_0^{\tau_1}|\nabla^\parallel_\sigma F|^2\,d\sigma =
        I(F)^2+
        2\int_0^T|\nabla^\parallel_\sigma F|^2\,d\sigma              \\
        &=
        I(F)^2+2|\nabla^{\mathcal H}F|_{\mathcal H}^2.
    \end{align*}
    Substituting this into \eqref{eq:path-minkowski}, then into \eqref{eq:path-first-step}, and taking expectations proves \eqref{eq:path-space-bobkov}. The induction is complete.
    
    Sharpness follows from the one-time cylinder case in the static Euclidean flow. There the inequality reduces to the sharp finite-dimensional Gaussian Bobkov inequality for \(N(0,2\tau)\). Since the one-time \(\mathcal H\)-energy is \(\tau|\nabla f|^2\), any coefficient smaller than \(2\) would contradict the sharp finite-dimensional Gaussian constant.
\end{proof}

The path-space isoperimetric inequality follows from the Bobkov inequality by the same relaxation argument used in the finite-dimensional case, see Theorem \ref{thm:intro-GI}. For a Borel set \(A\subset P_{(x,T)}M\), define its \(\mathcal H\)-perimeter by
\begin{align*}
    \operatorname{Per}^{\mathcal H}_{\Gamma}(A)
    =
    \inf
    \left\{
    \liminf_{j\to\infty}
    \int
    |\nabla^{\mathcal H}F_j|_{\mathcal H}\,d\Gamma_{(x,T)}
    \right\},
\end{align*}
where the infimum is over smooth cylinder functions \(0\le F_j\le1\) such that \(F_j\to \mathbf 1_A\) in \(L^1(\Gamma_{(x,T)}).\)

\begin{corollary}[Path-space Gaussian isoperimetry]
\label{cor:path-space-isoperimetry}
    For every Borel set \(A\subset P_{(x,T)}M\),
    \begin{align}
        \operatorname{Per}^{\mathcal H}_{\Gamma}(A)
        \ge
        \frac1{\sqrt2}I(\Gamma_{(x,T)}(A)).
        \label{eq:path-space-isoperimetry}
    \end{align}
\end{corollary}

\begin{remark}
    For one-time cylinder sets, Corollary~\ref{cor:path-space-isoperimetry} recovers Theorem~\ref{thm:intro-GI}. Thus the path-space inequality extends the sharp Gaussian isoperimetric inequality for conjugate heat-kernel measures to the parabolic Wiener space.
\end{remark}

\begin{remark}
    Haslhofer and Naber prove path-space log-Sobolev and spectral-gap estimates directly in \cite{MR3790068}. The path-space Bobkov inequality above is a corresponding isoperimetric statement in the same path-space framework.
\end{remark}


\bibliographystyle{amsalpha}
\bibliography{references}

@article {MR1374200,
    AUTHOR = {Bakry, D. and Ledoux, M.},
     TITLE = {L\'evy-{G}romov's isoperimetric inequality for an
              infinite-dimensional diffusion generator},
   JOURNAL = {Invent. Math.},
  FJOURNAL = {Inventiones Mathematicae},
    VOLUME = {123},
      YEAR = {1996},
    NUMBER = {2},
     PAGES = {259--281},
      ISSN = {0020-9910,1432-1297},
   MRCLASS = {58G32 (47D07 53C21)},
  MRNUMBER = {1374200},
MRREVIEWER = {Ming\ Liao},
       DOI = {10.1007/s002220050026},
       URL = {https://doi.org/10.1007/s002220050026},
}

@article{bamler2020entropy,
  title={Entropy and heat kernel bounds on a Ricci flow background},
  author={Bamler, Richard H},
  journal={arXiv preprint arXiv:2008.07093},
  pages={1-48},
  year={2020}
}

@article{bamler2020structure,
  title={Structure theory of non-collapsed limits of Ricci flows},
  author={Bamler, Richard H},
  journal={arXiv preprint arXiv:2009.03243},
  year={2020},
  pages={1-155}
}

@article {MR1428506,
    AUTHOR = {Bobkov, S. G.},
     TITLE = {An isoperimetric inequality on the discrete cube, and an
              elementary proof of the isoperimetric inequality in {G}auss
              space},
   JOURNAL = {Ann. Probab.},
  FJOURNAL = {The Annals of Probability},
    VOLUME = {25},
      YEAR = {1997},
    NUMBER = {1},
     PAGES = {206--214},
      ISSN = {0091-1798,2168-894X},
   MRCLASS = {60E15 (60G15)},
  MRNUMBER = {1428506},
MRREVIEWER = {Werner\ Linde},
       DOI = {10.1214/aop/1024404285},
       URL = {https://doi.org/10.1214/aop/1024404285},
}

@article {MR399402,
    AUTHOR = {Borell, Christer},
     TITLE = {The {B}runn-{M}inkowski inequality in {G}auss space},
   JOURNAL = {Invent. Math.},
  FJOURNAL = {Inventiones Mathematicae},
    VOLUME = {30},
      YEAR = {1975},
    NUMBER = {2},
     PAGES = {207--216},
      ISSN = {0020-9910,1432-1297},
   MRCLASS = {28A40 (60G15)},
  MRNUMBER = {399402},
MRREVIEWER = {A.\ Badrikian},
       DOI = {10.1007/BF01425510},
       URL = {https://doi.org/10.1007/BF01425510},
}

@article{gross-1975-log-sobolev,
    AUTHOR = {Gross, Leonard},
     TITLE = {Logarithmic {S}obolev inequalities},
   JOURNAL = {Amer. J. Math.},
  FJOURNAL = {American Journal of Mathematics},
    VOLUME = {97},
      YEAR = {1975},
    NUMBER = {4},
     PAGES = {1061--1083},
      ISSN = {0002-9327,1080-6377},
   MRCLASS = {46E35 (81.47)},
  MRNUMBER = {420249},
MRREVIEWER = {R.\ H\o egh-Krohn},
       DOI = {10.2307/2373688},
       URL = {https://doi.org/10.2307/2373688},
}

@article {MR664497,
    AUTHOR = {Hamilton, Richard S.},
     TITLE = {Three-manifolds with positive {R}icci curvature},
   JOURNAL = {J. Differential Geometry},
  FJOURNAL = {Journal of Differential Geometry},
    VOLUME = {17},
      YEAR = {1982},
    NUMBER = {2},
     PAGES = {255--306},
      ISSN = {0022-040X,1945-743X},
   MRCLASS = {53C25 (35K55 58G30)},
  MRNUMBER = {664497},
MRREVIEWER = {J.\ L.\ Kazdan},
       URL = {http://projecteuclid.org/euclid.jdg/1214436922},
}

@book {MR46395,
    AUTHOR = {Hardy, G. H. and Littlewood, J. E. and P\'olya, G.},
     TITLE = {Inequalities},
      NOTE = {2d ed},
 PUBLISHER = {Cambridge, at the University Press, },
      YEAR = {1952},
     PAGES = {xii+324},
   MRCLASS = {27.0X},
  MRNUMBER = {46395},
}

@article {MR3790068,
    AUTHOR = {Haslhofer, Robert and Naber, Aaron},
     TITLE = {Characterizations of the {R}icci flow},
   JOURNAL = {J. Eur. Math. Soc. (JEMS)},
  FJOURNAL = {Journal of the European Mathematical Society (JEMS)},
    VOLUME = {20},
      YEAR = {2018},
    NUMBER = {5},
     PAGES = {1269--1302},
      ISSN = {1435-9855,1435-9863},
   MRCLASS = {53C44 (53C21)},
  MRNUMBER = {3790068},
MRREVIEWER = {Sylvain\ Maillot},
       DOI = {10.4171/JEMS/787},
       URL = {https://doi.org/10.4171/JEMS/787},
}

@article {MR3245102,
    AUTHOR = {Hein, Hans-Joachim and Naber, Aaron},
     TITLE = {New logarithmic {S}obolev inequalities and an
              {$\epsilon$}-regularity theorem for the {R}icci flow},
   JOURNAL = {Comm. Pure Appl. Math.},
  FJOURNAL = {Communications on Pure and Applied Mathematics},
    VOLUME = {67},
      YEAR = {2014},
    NUMBER = {9},
     PAGES = {1543--1561},
      ISSN = {0010-3640,1097-0312},
   MRCLASS = {35A23 (35R01 53C23 53C44)},
  MRNUMBER = {3245102},
MRREVIEWER = {Yu\ Ding},
       DOI = {10.1002/cpa.21474},
       URL = {https://doi.org/10.1002/cpa.21474},
}

@article{perelman2002entropy,
  title={The entropy formula for the Ricci flow and its geometric applications},
  author={Perelman, Grisha},
  journal={arXiv preprint math/0211159},
  pages={1-2},
  year={2002}
}

@book {MR2265633,
    AUTHOR = {Shaked, Moshe and Shanthikumar, J. George},
     TITLE = {Stochastic orders},
    SERIES = {Springer Series in Statistics},
 PUBLISHER = {Springer, New York},
      YEAR = {2007},
     PAGES = {xvi+473},
      ISBN = {978-0-387-32915-4; 0-387-32915-3},
   MRCLASS = {60-02 (60E15 62H05 62N05)},
  MRNUMBER = {2265633},
MRREVIEWER = {B.\ L. S. Prakasa Rao},
       DOI = {10.1007/978-0-387-34675-5},
       URL = {https://doi.org/10.1007/978-0-387-34675-5},
}

@book {MR4573029,
    AUTHOR = {Stroock, Daniel W.},
     TITLE = {Gaussian measures in finite and infinite dimensions},
    SERIES = {Universitext},
 PUBLISHER = {Springer, Cham},
      YEAR = {[2023] \copyright 2023},
     PAGES = {xii+144},
      ISBN = {978-3-031-23121-6; 978-3-031-23122-3},
   MRCLASS = {60-02 (28C20 46G12 60B11 60G15)},
  MRNUMBER = {4573029},
MRREVIEWER = {Ramon\ van Handel},
       DOI = {10.1007/978-3-031-23122-3},
       URL = {https://doi.org/10.1007/978-3-031-23122-3},
}

@article {MR365680,
    AUTHOR = {Sudakov, V. N. and Cirelson, B. S.},
     TITLE = {Extremal properties of half-spaces for spherically invariant
              measures},
      NOTE = {Problems in the theory of probability distributions, II},
   JOURNAL = {Zap. Nau\v cn. Sem. Leningrad. Otdel. Mat. Inst. Steklov.
              (LOMI)},
  FJOURNAL = {Zapiski Nau\v cnyh Seminarov Leningradskogo Otdelenija
              Matemati\v ceskogo Instituta im. V. A. Steklova Akademii Nauk
              SSSR (LOMI)},
    VOLUME = {41},
      YEAR = {1974},
     PAGES = {14--24, 165},
   MRCLASS = {60G15},
  MRNUMBER = {365680},
MRREVIEWER = {R.\ M.\ Dudley},
}

\end{document}